\documentclass[hidelinks,onefignum,onetabnum]{siamart250211}

\usepackage{lipsum}
\usepackage{amsfonts}
\usepackage{graphicx}
\usepackage{epstopdf}
\usepackage{algorithmic}
\ifpdf
  \DeclareGraphicsExtensions{.eps,.pdf,.png,.jpg}
\else
  \DeclareGraphicsExtensions{.eps}
\fi

\newsiamremark{remark}{Remark}
\newsiamremark{hypothesis}{Hypothesis}
\crefname{hypothesis}{Hypothesis}{Hypotheses}
\newsiamthm{claim}{Claim}
\newsiamremark{fact}{Fact}
\crefname{fact}{Fact}{Facts}
\newsiamthm{assumption}{Assumption}

\headers{Mini-Extragradient Methods}{Xiaozhi Liu and Yong Xia}

\title{Mini-Extragradient Methods\thanks{\funding{This work was funded by by National Key Research and Development Program of China under grant 2021YFA1003303 and National Natural Science Foundation of China under grant 12171021.}}}

\author{Xiaozhi Liu\thanks{School of Mathematical Sciences, Beihang University, Beijing 100191, People's Republic of China, 
		\{xzliu, yxia\}@buaa.edu.cn.}
	\and Yong Xia\footnotemark[2]
	\thanks{Corresponding author. }
}

\usepackage{amsopn}

\usepackage{color}

\usepackage{colortbl}
\definecolor{bblue}{rgb}{0.855,0.933,0.98}

\definecolor{rred}{HTML}{DC143C}

\usepackage{multirow}

\usepackage{makecell}

\usepackage{booktabs}

\usepackage{subcaption} 

\usepackage{rotating}

\ifpdf
\hypersetup{
  pdftitle={Mini-Extragradient Methods},
  pdfauthor={Xiaozhi Liu and Yong Xia}
}
\fi

\begin{document}

\maketitle

\begin{abstract}
The Extragradient (EG) method stands as a cornerstone algorithm for solving monotone nonlinear equations but faces two important unresolved challenges: (i) how to select stepsizes without relying on the global Lipschitz constant or expensive line-search procedures, and (ii) how to reduce the two full evaluations of the mapping required per iteration to effectively one, without compromising convergence guarantees or computational efficiency. To address the first challenge, we propose the Greedy Mini-Extragradient (Mini-EG) method, which updates only the coordinate associated with the dominant component of the mapping at each extragradient step. This design capitalizes on componentwise Lipschitz constants that are far easier to estimate than the classical global Lipschitz constant. To further lower computational cost, we introduce a Random Mini-EG variant that replaces full mapping evaluations by sampling only a single coordinate per extragradient step. Although this resolves the second challenge from a theoretical standpoint, its practical efficiency remains limited. To bridge this gap, we develop the Watchdog-Max strategy, motivated by the slow decay of dominant component magnitudes. Instead of evaluating the full mapping, Watchdog-Max identifies and tracks only two coordinates at each extragradient step, dramatically reducing per-iteration cost while retaining strong practical performance. We establish convergence guarantees and rate analyses for all proposed methods. In particular, Greedy Mini-EG achieves enhanced convergence rates that surpass the classical guarantees of the vanilla EG method in several standard application settings. Numerical experiments on regularized decentralized logistic regression and compressed sensing show speedups exceeding $13\times$ compared with the classical EG method on both synthetic and real datasets.
\end{abstract}



\section{Introduction}

Many fundamental problems in optimization and scientific computing, including min-max problem and variational inequalities, can be cast as monotone nonlinear equations. These formulations underpin a broad range of applications, such as decentralized logistic regression \cite{jian2022family}, generative adversarial networks \cite{goodfellow2014generative}, image processing \cite{figueiredo2008gradient}, and robust learning \cite{wen2014robust}.
However, solving them efficiently remains a significant challenge, especially in regimes where the feature dimension substantially exceeds the number of samples, a scenario that has become increasingly prominent in modern few-sample scenarios \cite{wang2020generalizing}.

One of the most widely studied methods for solving monotone nonlinear equations is the Extragradient (EG) method \cite{korpelevich1976extragradient}, which guarantees convergence under $L$-Lipschitz continuity. It has inspired a broad class of EG-type methods, each enhancing performance through different strategies. For example, EG$+$ \cite{diakonikolas2021efficient} adopts a two-time-scale framework that allows more flexible stepsize selection. The Extra-Anchored-Gradient (EAG) method \cite{yoon2021accelerated} achieves faster convergence by combining an initial and the most recent iterate at each step, achieving the optimal rate under monotonicity and Lipschitz continuity assumptions. Building on both EG$+$ and EAG, the Fast Extragradient (FEG) method \cite{lee2021fast} further integrates two-time-scale and anchoring techniques to further accelerate convergence. Another important variant is the Past Extragradient (PEG) method \cite{popov1980modification}, also known as the Optimistic Gradient method, which reduces computational cost by storing and reusing historical iterates.

Despite their broad applicability, EG-type methods suffer from two fundamental limitations:
(i) Their stepsize rules depend on the global Lipschitz constant $L$, which is often unavailable or expensive to estimate in practice. As a result, practitioners typically resort to manual tuning or computationally intensive line-search procedures \cite{xiao2011non}, undermining efficiency and practicality. 
(ii) Each iteration requires either two full evaluations of the mapping or storage of the entire previous iterate. Both options incur significant computational and memory costs, especially in large-scale settings.

These challenges naturally lead to two fundamental and unresolved questions:
\begin{enumerate}
	\item Can we design a line-search-free EG variant that requires no knowledge of the global Lipschitz constant $L$, yet still enjoys rigorous convergence guarantees?
	\item Can such a variant retain its convergence properties while reducing the per-iteration cost from two full evaluations of the mapping to effectively one, thereby improving overall computational efficiency?
\end{enumerate}
Motivated by these questions, we develop three Mini-Extragradient (Mini-EG) methods: Greedy Mini-EG, Random Mini-EG, and Watchdog-Max.
These methods eliminate the dependence on the global Lipschitz constant $L$ for stepsize selection and updates only a single coordinate at each extragradient step, thereby avoiding full evaluations of the mapping.

Our main contributions are summarized as follows:
\begin{itemize}
	\item These Mini-EG variants determine stepsizes based solely on componentwise Lipschitz constants, which are typically easier to compute and smaller than the global Lipschitz constant $L$.
	\item A key contribution of this work is Watchdog-Max, a novel Mini-EG method inspired by the slow decay of dominant component magnitudes. By incorporating randomized sampling, it requires only two coordinate evaluations per extragradient step, significantly reducing computational cost while maintaining convergence performance comparable to the greedy selection strategy.
	\item We provide rigorous convergence analysis and rate guarantees for all proposed methods. Notably, in several standard application settings, the Greedy Mini-EG method attains enhanced convergence rates that go beyond the classical guarantees established for the vanilla EG method.
	\item We evaluate the proposed methods on two representative tasks: regularized decentralized logistic regression and compressed sensing. On both synthetic and real datasets, the methods achieve speedups exceeding ${13\times}$ relative to the standard EG algorithm.
\end{itemize}

The remainder of the paper is organized as follows. \Cref{sec:preliminary} introduces the problem formulation, basic assumptions, and a brief review of the EG method. In \cref{sec:minieg}, we present three novel Mini-EG methods and establish convergence analysis and rate guarantees. Extensive numerical experiments in \cref{sec:experiment} demonstrate substantial acceleration of the proposed methods over the standard EG algorithm on both regularized decentralized logistic regression and compressed sensing tasks. Finally, \cref{sec:conclusion} concludes the paper and discusses several potential directions for future research.

\noindent \textbf{Notation.}
Matrices are denoted by bold uppercase letters $\boldsymbol{A}$, vectors by bold lowercase letters $\boldsymbol{a}$, and scalars by lowercase letters $a$; $\boldsymbol{A}^T$ is the transpose of $\boldsymbol{A}$, and $\|\boldsymbol{a}\|$, $\|\boldsymbol{a}\|_{1}$ and $\|\boldsymbol{a}\|_{\infty}$ denote the $\ell_2$-, $\ell_1$-, and $\ell_\infty$-norms of $\boldsymbol{a}$, respectively.

\section{Preliminaries}\label{sec:preliminary}

In this paper, we consider the problem of solving the following constrained nonlinear equations:
\begin{equation}
	F(\boldsymbol{x}) = \boldsymbol{0},\quad \boldsymbol{x} \in \Omega,
	\label{eq:mono_equations}
\end{equation}
where $F: \mathbb{R}^n \rightarrow \mathbb{R}^n$ is a mapping of the form $F(\boldsymbol{x}) = [F_1(\boldsymbol{x}), \dots, F_n(\boldsymbol{x})]^T$, and $\Omega \subseteq \mathbb{R}^n$ is a nonempty, closed, and convex set.

Most existing works assume that the mapping $F$ is monotone and $L$-Lipschitz continuous, as defined below:
\begin{assumption} \label{assumption:monotone}
	The mapping $F: \mathbb{R}^n \rightarrow \mathbb{R}^n$ is monotone, i.e.,
	$$
	\left\langle F(\boldsymbol{x}) - F(\boldsymbol{y}), \boldsymbol{x} - \boldsymbol{y}	\right\rangle \geq 0, \quad \forall \boldsymbol{x}, \boldsymbol{y} \in \mathbb{R}^n.
	$$
\end{assumption}

\begin{assumption} \label{assumption:Lipschitz}
	The mapping $F: \mathbb{R}^n \rightarrow \mathbb{R}^n$ is $L$-Lipschitz continuous, i.e., there exists a constant $L$ such that
	$$
	\left\| F(\boldsymbol{x}) - F(\boldsymbol{y}) \right\| \leq L \left\| \boldsymbol{x} - \boldsymbol{y} \right\|,\quad \forall \boldsymbol{x}, \boldsymbol{y} \in \mathbb{R}^n.
	$$
\end{assumption}

Introduced by \cite{korpelevich1976extragradient}, the EG method is one of the most established approaches for solving monotone equations as in \cref{eq:mono_equations}.
Starting from an initial point $\boldsymbol{x}_0 \in \mathbb{R}^n$, the method alternates between an extragradient step and an update step, as described below:
$$
\begin{aligned}
	\boldsymbol{y}_k &= \boldsymbol{x}_k - \alpha F(\boldsymbol{x}_k), \\
	\boldsymbol{x}_{k+1} &= P_{\Omega} \left(\boldsymbol{x}_k - \alpha F(\boldsymbol{y}_k)\right),
\end{aligned}
$$
where $k$ is the iteration index, $\alpha > 0$ is the stepsize, and $P_{\Omega}$ denotes the projection operator onto the set $\Omega$.
In practice, the choice of $\alpha$ is critical, as it directly affects both convergence speed and numerical stability.

To leverage a more advanced and flexible benchmark, we adopt a two-time-scale EG variant in which the update step uses an adaptive stepsize $\beta_k$. This modification alleviates the rigid dependence on the global Lipschitz constant required by the classical EG method:
\begin{equation}
	\boxed{
		\begin{aligned}
			\boldsymbol{y}_k &= \boldsymbol{x}_k - \alpha F(\boldsymbol{x}_k), \\
			\boldsymbol{x}_{k+1} &= P_{\Omega} \left(\boldsymbol{x}_k - \beta_k F(\boldsymbol{y}_k)\right),
		\end{aligned}
	}
	\label{eq:eg}
\end{equation}
where $\alpha = \rho / L$ with a fixed parameter $\rho \in (0,1)$. The stepsize $\beta_k$ is computed as
\begin{equation}
	\beta_k = \frac{F(\boldsymbol{y}_k)^T\left(\boldsymbol{x}_k - \boldsymbol{y}_k\right)}{\|F(\boldsymbol{y}_k)\|^2}.
	\label{eq:beta}
\end{equation}

Despite its popularity, the EG method has two notable limitations:

\begin{itemize}
	\item [(i)] The convergence of EG requires the stepsize $\alpha$ to satisfy $0 < \alpha < 1/L$, where $L$ is the global Lipschitz constant of the mapping $F$ \cite{korpelevich1976extragradient}. In practice, $L$ is often difficult to compute or estimate reliably. As a result, manual tuning or line-search procedures are typically employed to select a suitable stepsize, which introduces additional computational overhead and reduces practicality.
	\item [(ii)] Each iteration of EG requires two full evaluations of the mapping $F$, which can be computationally intensive or even prohibitive in large-scale problems.
\end{itemize}

In this paper, we propose a strategy that eliminates the stepsize dependency on the global Lipschitz constant $L$ and requires only a single full evaluation of the mapping $F$ per iteration. This strategy is based on the assumption that $F$ is componentwise Lipschitz continuous \cite{nesterov2012efficiency}, which is a strictly weaker condition than \cref{assumption:Lipschitz}. The formal definition is given below:

\begin{definition}
	A mapping $F: \mathbb{R}^n \rightarrow \mathbb{R}^n$ is said to be componentwise Lipschitz continuous if there exist constants $l_1, \dots, l_n > 0$ such that for all $i = 1, \dots, n$,
	\begin{equation}
		|F_i(\boldsymbol{x}+te_i) - F_i(\boldsymbol{x})| \leq l_i |t|, \quad \forall \boldsymbol{x} \in \mathbb{R}^n, t \in \mathbb{R},
		\label{eq:def_comp_lip}
	\end{equation}
	where $e_i$ denotes the $i$-th standard basis vector in $\mathbb{R}^n$.
\end{definition}

\begin{remark} \label{remk:lip_comp}
	It can be readily verified that if $F$ is $L$-Lipschitz continuous, then it is also componentwise Lipschitz continuous, with constants $l_i \leq L$ for all $i = 1, \dots, n$. Moreover, the constants $l_i$ are often easier to compute, especially in large-scale problems. We demonstrate this observation in the numerical experiments section.
\end{remark}

Throughout the paper, we assume that both Assumptions~\ref{assumption:monotone} and \ref{assumption:Lipschitz} hold. In addition, we further assume that the mapping $F$ satisfies the following condition:

\begin{assumption} \label{assumption:comp_lip}
	The mapping $F: \mathbb{R}^n \rightarrow \mathbb{R}^n$ is componentwise Lipschitz continuous with constants $l_1, \dots, l_n > 0$.
\end{assumption}

As discussed in \cref{remk:lip_comp}, \cref{assumption:comp_lip} is implied by \cref{assumption:Lipschitz}. Therefore, our analysis introduces no additional assumptions.

\begin{remark}
	\Cref{assumption:Lipschitz} is imposed solely for the purpose of convergence analysis. In practice, the proposed algorithms do not require access to the global Lipschitz constant $L$; instead, they rely only on the componentwise constants $l_1,\dots, l_n$, which are substantially easier to compute and are generally much smaller than the global constant $L$.
\end{remark}

\section{Mini-Extragradient Methods and Convergence Analysis} \label{sec:minieg}

Compared with the gradient descent method, the EG method offers the notable advantage of guaranteeing convergence under only Assumptions~\ref{assumption:monotone} and \ref{assumption:Lipschitz}. In this work, we further demonstrate that convergence can still be ensured even when the extragradient step updates only a single coordinate at each iteration, offering a more efficient alternative for high-dimensional problems.
We refer to these variants as Mini-Extragradient (Mini-EG) methods.

\subsection{Greedy Mini-EG}

The simplest and most intuitive coordinate selection approach is the greedy scheme, resulting in a variant we refer to as Greedy Mini-EG (G-Mini-EG):
\begin{equation}
	\boxed{
		\begin{aligned}
			i_k &= \arg\max_{1\leq i \leq n} |F_i(\boldsymbol{x}_k)|, \\
			\boldsymbol{y}_k &= \boldsymbol{x}_k - \frac{\rho}{l_{i_k}} F_{i_k}(\boldsymbol{x}_k) e_{i_k}, \\
			\boldsymbol{x}_{k+1} &= P_{\Omega} \left(\boldsymbol{x}_k - \beta_k F(\boldsymbol{y}_k)\right),
		\end{aligned}
	}
	\label{eq:greedy_EG}
\end{equation}
where $\rho\in (0,1)$ is a fixed parameter, and the stepsize $\beta_k$ is identical to the EG stepsize rule in \cref{eq:beta}.

\begin{remark}
	(i) The proposed procedure relies only on the componentwise Lipschitz constants $l_1, \dots, l_n$, which are typically much easier to compute and substantially smaller than the global constant $L$.
	
	(ii) The computation of $\beta_k$ can be simplified due to the componentwise update structure. Compared with the EG method in \cref{eq:beta}, it avoids full vector products, and the resulting expression is
	\begin{equation}
		\beta_k = \frac{\rho F_{i_k}(\boldsymbol{y}_k) F_{i_k}(\boldsymbol{x}_k)}{l_{i_k} \|F(\boldsymbol{y}_k)\|^2}.
		\label{eq:beta_comp}
	\end{equation}
\end{remark}

We now proceed to establish the convergence of the G-Mini-EG algorithm.

It is well known that the projection operator onto a convex set is nonexpansive; see, for instance, \cite{bauschke2011convex}. For completeness, we restate this fundamental property below.

\begin{lemma} \label{lem:proj_monotone}
	let $\Omega \subseteq \mathbb{R}^n$ be a nonempty, closed, and convex set. Then the projection operator $P_{\Omega}$ is 1-Lipschitz continuous; that is,
	$$
	\left\|P_{\Omega}(\boldsymbol{x})-P_{\Omega}(\boldsymbol{y})\right\| \leq\|\boldsymbol{x} - \boldsymbol{y}\|,\quad \forall \boldsymbol{x}, \boldsymbol{y} \in \mathbb{R}^n.
	$$
\end{lemma}

For the general mini-extragradient step
\begin{equation}
	\boldsymbol{y} = \boldsymbol{x} - \frac{\rho}{l_{i}} F_{i}(\boldsymbol{x}) e_{i},\quad \forall i = 1, \dots, n,
	\label{eq:mini-extra}
\end{equation}
we establish the following results:
\begin{lemma} \label{lem:Fy}
	Suppose that \cref{assumption:Lipschitz} holds. For the update in \cref{eq:mini-extra}, it holds that
	$$
	\|F(\boldsymbol{y})\| \leq \|F(\boldsymbol{x})\| + \frac{\rho L}{l_i} |F_i(\boldsymbol{x})|,\quad \forall i = 1, \dots, n.
	$$
\end{lemma}

\begin{proof}
	By \cref{assumption:Lipschitz} and the update in \cref{eq:mini-extra}, for any $i = 1, \dots, n$, we have
	$$
	\|F(\boldsymbol{y}) - F(\boldsymbol{x})\| \le L \|\boldsymbol{y} - \boldsymbol{x}\| = \frac{\rho L}{l_i} |F_i(\boldsymbol{x})|.
	$$
	The result follows immediately by the triangle inequality.
\end{proof}

To facilitate the subsequent analysis, we introduce the following relaxed bound.

\begin{corollary}\label{cor:Fy}
	Under the same assumption as in \cref{lem:Fy}, the update \cref{eq:mini-extra} satisfies
	$$
	\|F(\boldsymbol{y})\| \le \left(1+\frac{\rho L}{l_i}\right)\|F(\boldsymbol{x})\|,\quad \forall i = 1, \dots, n.
	$$
\end{corollary}

\begin{lemma} \label{lem:product}
	Under \cref{assumption:comp_lip}, consider the update in \cref{eq:mini-extra}.
	Then, for any $i = 1, \dots, n$, it holds that
	$$
	F(\boldsymbol{y})^T\left(\boldsymbol{x} - \boldsymbol{y}\right) \geq \frac{\rho(1-\rho)}{l_i} |F_i(\boldsymbol{x})|^2.
	$$
\end{lemma}

\begin{proof}
	Given \cref{assumption:comp_lip} and the form of the update in \cref{eq:mini-extra}, it follows that for any $i = 1, \dots, n$:
	$$
	\begin{aligned}
		\left(F(\boldsymbol{x}) - F(\boldsymbol{y})\right)^T\left(\boldsymbol{x} - \boldsymbol{y}\right) &= \frac{\rho}{l_i} \left(F_i(\boldsymbol{x}) - F_i(\boldsymbol{y})\right)F_i(\boldsymbol{x}) \\
		&\leq \frac{\rho}{l_i} |F_i(\boldsymbol{x}) - F_i(\boldsymbol{y})| |F_i(\boldsymbol{x})| \\
		&\leq \frac{\rho^2}{l_i} |F_i(\boldsymbol{x})|^2,
	\end{aligned}
	$$
	where the last inequality follows from the componentwise Lipschitz condition defined in \cref{eq:def_comp_lip}.
	
	Rearranging terms, we obtain
	$$
	\begin{aligned}
		F(\boldsymbol{y})^T\left(\boldsymbol{x} - \boldsymbol{y}\right) &\geq F(\boldsymbol{x})^T\left(\boldsymbol{x} - \boldsymbol{y}\right) - \frac{\rho^2}{l_i} |F_i(\boldsymbol{x})|^2 \\
		&= \frac{\rho}{l_i} |F_i(\boldsymbol{x})|^2 - \frac{\rho^2}{l_i} |F_i(\boldsymbol{x})|^2 \\
		&= \frac{\rho(1-\rho)}{l_i} |F_i(\boldsymbol{x})|^2,
	\end{aligned}
	$$
	which completes the proof.
\end{proof}

let $\boldsymbol{x}^* \in \Omega$ be a solution to problem \cref{eq:mono_equations}, i.e., $F(\boldsymbol{x}^*) = \boldsymbol{0}$, and define $l_{\max} = \max_{1\leq i \leq n} l_i$. With the preparatory results established above, we are ready to present the global convergence analysis of the G-Mini-EG method.

\begin{lemma} \label{lem:greedy_EG_convergence_point}
	Suppose that Assumptions~\ref{assumption:monotone}, \ref{assumption:Lipschitz}, and \ref{assumption:comp_lip} hold. Let $\{\boldsymbol{x}_k\}$ and $\{\boldsymbol{y}_k\}$ be the sequence generated by the method \cref{eq:greedy_EG}. Then, for all $k\geq 0$, it holds that
	\begin{equation}
		\|\boldsymbol{x}_{k} - \boldsymbol{x}^*\|^2 - \|\boldsymbol{x}_{k+1} - \boldsymbol{x}^*\|^2 \geq \frac{\rho^2(1-\rho)^2}{(\sqrt{n}l_{\max} + \rho L)^2} \|F(\boldsymbol{x}_k)\|_{\infty}^2.
		\label{eq:greed_EG_convergence_point}
	\end{equation}
\end{lemma}

\begin{proof}
	From the update rule in \cref{eq:greedy_EG} and \cref{lem:proj_monotone}, we obtain
	\begin{equation}
		\begin{aligned}
			\|\boldsymbol{x}_{k+1} - \boldsymbol{x}^*\|^2 =& \|P_{\Omega} \left(\boldsymbol{x}_k - \beta_k F(\boldsymbol{y}_k)\right) - P_{\Omega} \left(\boldsymbol{x}^*\right)\|^2 \\
			\leq& \|\boldsymbol{x}_k - \boldsymbol{x}^* - \beta_k F(\boldsymbol{y}_k)\|^2 \\
			=& \|\boldsymbol{x}_k - \boldsymbol{x}^*\|^2 + \beta_k^2 \|F(\boldsymbol{y}_k)\|^2 - 2 \beta_k F(\boldsymbol{y}_k)^T (\boldsymbol{x}_k - \boldsymbol{x}^*).
		\end{aligned}
		\label{eq:greed_EG_bound_point}
	\end{equation}
	By the monotonicity of the mapping $F$ (\cref{assumption:monotone}), we have
	\begin{equation}
		\begin{aligned}
			F(\boldsymbol{y}_k)^T (\boldsymbol{x}_k - \boldsymbol{x}^*) =& F(\boldsymbol{y}_k)^T (\boldsymbol{x}_k - \boldsymbol{y}_k) + \left(F(\boldsymbol{y}_k) - F(\boldsymbol{x}^*)\right)^T (\boldsymbol{y}_k - \boldsymbol{x}^*) \\
			\geq& F(\boldsymbol{y}_k)^T (\boldsymbol{x}_k - \boldsymbol{y}_k).
		\end{aligned}
		\label{eq:greed_EG_monotone}
	\end{equation}
	Substituting \cref{eq:greed_EG_monotone} into \cref{eq:greed_EG_bound_point} yields
	$$
	\|\boldsymbol{x}_{k+1} - \boldsymbol{x}^*\|^2 \leq \|\boldsymbol{x}_k - \boldsymbol{x}^*\|^2 + \beta_k^2 \|F(\boldsymbol{y}_k)\|^2 - 2 \beta_k F(\boldsymbol{y}_k)^T (\boldsymbol{x}_k - \boldsymbol{y}_k).
	$$
	Using the definition of $\beta_k$ in \cref{eq:beta}, we further obtain
	$$
	\|\boldsymbol{x}_{k+1} - \boldsymbol{x}^*\|^2 \leq \|\boldsymbol{x}_k - \boldsymbol{x}^*\|^2 - \frac{\left(F(\boldsymbol{y}_k)^T (\boldsymbol{x}_k - \boldsymbol{y}_k)\right)^2}{\|F(\boldsymbol{y}_k)\|^2}.
	$$
	Finally, invoking \cref{lem:Fy,lem:product} yields
	$$
	\|\boldsymbol{x}_{k} - \boldsymbol{x}^*\|^2 - \|\boldsymbol{x}_{k+1} - \boldsymbol{x}^*\|^2 \geq \left(\frac{\rho(1-\rho)|F_{i_k}(\boldsymbol{x}_k)|^2}{l_{i_k} \|F(\boldsymbol{x}_k)\| + \rho L |F_{i_k}(\boldsymbol{x}_k)|}\right)^2.
	$$
	Under the greedy selection rule in \cref{eq:greedy_EG}, which ensures $|F_{i_k}(\boldsymbol{x}_k)| = \|F(\boldsymbol{x}_k)\|_{\infty}$, and using the relation $\|F(\boldsymbol{x}_k)\| \leq \sqrt{n}\, \|F(\boldsymbol{x}_k)\|_{\infty}$, we obtain
	$$
	\begin{aligned}
		\|\boldsymbol{x}_{k} - \boldsymbol{x}^*\|^2 - \|\boldsymbol{x}_{k+1} - \boldsymbol{x}^*\|^2
		&\geq \frac{\rho^2(1-\rho)^2}{(\sqrt{n}l_{i_k} + \rho L)^2} \|F(\boldsymbol{x}_k)\|_{\infty}^2\\
		&\geq \frac{\rho^2(1-\rho)^2}{(\sqrt{n}l_{\max} + \rho L)^2} \|F(\boldsymbol{x}_k)\|_{\infty}^2,
	\end{aligned}
	$$
	where the last inequality follows from $l_{i_k}\leq l_{\max}$.
\end{proof}

\begin{theorem}	\label{thm:greedy_EG_convergence_F}
	Suppose Assumptions~\ref{assumption:monotone}, \ref{assumption:Lipschitz}, and \ref{assumption:comp_lip} hold, and let the sequences $\{\boldsymbol{x}_k\}$ and $\{\boldsymbol{y}_k\}$ be generated by the method \cref{eq:greedy_EG}. Then, for any integer $K\geq 0$, we have
	\begin{equation}
		\min_{0\leq k \leq K} \|F(\boldsymbol{x}_k)\|_{\infty}^2 \leq \frac{(\sqrt{n}l_{\max} + \rho L)^2}{\rho^2(1-\rho)^2(K+1)} \|\boldsymbol{x}_{0} - \boldsymbol{x}^*\|^2.
		\label{eq:greed_EG_convergence_func}
	\end{equation}
	In addition, the residual sequence vanishes asymptotically:
	\begin{equation}
		\liminf_{k\rightarrow \infty} \|F(\boldsymbol{x}_k)\|_{\infty} = 0.
		\label{eq:greed_EG_convergence_func_lim}
	\end{equation}
\end{theorem}

\begin{proof}
	Summing inequality \cref{eq:greed_EG_convergence_point} from $k=0$ to $K$, we obtain
	$$
	\begin{aligned}
		\frac{\rho^2(1-\rho)^2}{(\sqrt{n}l_{\max} + \rho L)^2} \sum_{k=0}^{K} \|F(\boldsymbol{x}_k)\|_{\infty}^2 &\leq \|\boldsymbol{x}_{0} - \boldsymbol{x}^*\|^2 - \|\boldsymbol{x}_{K+1} - \boldsymbol{x}^*\|^2 \\
		&\leq \|\boldsymbol{x}_{0} - \boldsymbol{x}^*\|^2.
	\end{aligned}
	$$
	On the other hand, the minimum is bounded by the average:
	$$
	\min_{0\leq k \leq K} \|F(\boldsymbol{x}_k)\|_{\infty}^2 \leq \frac{1}{K+1} \sum_{k=0}^{K} \|F(\boldsymbol{x}_k)\|_{\infty}^2.
	$$
	Combining the two inequalities yields \cref{eq:greed_EG_convergence_func}. Taking the limit as $K \rightarrow \infty$ gives \cref{eq:greed_EG_convergence_func_lim}, which completes the proof.
\end{proof}

Define the coupling factor $\kappa = {L}/{l_{\max}}$, which measures how much the global Lipschitz constant can exceed the directional derivative bounds along coordinate axes. We use this quantity to compare the convergence behavior of the G-Mini-EG method with that of the classical EG algorithm. It is well known that $\kappa \in [1, n]$ \cite{wright2015coordinate}.
A key observation is that, unlike the convergence guarantees for classical coordinate-descent methods \cite{nesterov2012efficiency}, the coefficient of $L$ in \cref{thm:greedy_EG_convergence_F} is independent of the problem dimension $n$. Moreover, in regimes where $\kappa > \sqrt{n}$, which frequently occur in least absolute shrinkage and selection operator (LASSO) problems \cite{tibshirani1996regression} (see \cref{eq:cs_L,eq:cs_li} in \cref{sec:cs}), the worst-case complexity of the G-Mini-EG method coincides with that of the standard EG algorithm.\footnote{For a random matrix with independent, zero-mean entries of finite variance, random matrix theory provides an average-case estimate~\cite{bandeira2016sharp}. The spectral norm (global Lipschitz constant) scales as $\Theta(\sqrt{n})$, whereas the maximum diagonal element typically scales only as $\Theta(\sqrt{\log n})$. Consequently, the expected coupling factor grows as $\Theta(\sqrt{n / \log n})$.}

\subsection{Random Mini-EG}

Although we have rigorously established the convergence of the G-Mini-EG, where only a single coordinate is updated at each extragradient step and the stepsize is determined solely by componentwise Lipschitz constants, the method still requires two full evaluations of the mapping $F$ per iteration to identify the dominant component. This becomes computationally prohibitive in large-scale settings.

To address this limitation, we introduce a more efficient strategy inspired by random coordinate descent \cite{nesterov2012efficiency}. At each iteration $k$, the index $i_k$ is sampled according to a probability distribution, leading to the following variant, which we refer to as the Random Mini-EG (R-Mini-EG):
\begin{equation}
	\boxed{
		\begin{aligned}
			i_k &= \mathcal{A}_{\gamma}, \\
			\boldsymbol{y}_k &= \boldsymbol{x}_k - \frac{\rho}{l_{i_k}} F_{i_k}(\boldsymbol{x}_k) e_{i_k}, \\
			\boldsymbol{x}_{k+1} &= P_{\Omega} \left(\boldsymbol{x}_k - \beta_k F(\boldsymbol{y}_k)\right),
		\end{aligned}
	}
	\label{eq:random_EG}
\end{equation}
Here, the operation $i_k = \mathcal{A}_{\gamma}$ indicates that the index $i_k \in \left\{1,\dots,n\right\}$ is drawn independently according to the probability distribution
$$
p_{\gamma}(i) = \frac{l_i^{\gamma}}{\sum_{j=1}^{n} l_j^{\gamma}},\quad \forall i = 1, \dots, n,
$$
with parameter $\gamma \geq 0$ is a tunable parameter. The scalar $\rho\in(0,1)$, and the stepsize $\beta_k$ is computed as in \cref{eq:beta_comp}.

To establish the convergence of the R-Mini-EG, we first introduce the following weighted norm \cite{nesterov2012efficiency}:
\begin{equation}
	\|\boldsymbol{x}\|_{[\gamma]}=\sqrt{\sum_{i=1}^n l_i^\gamma x_i^2}.
	\label{eq:weighted_norm}
\end{equation}

\begin{lemma} \label{lem:random_EG_convergence_point}
	Let Assumptions~\ref{assumption:monotone}, \ref{assumption:Lipschitz}, and \ref{assumption:comp_lip} hold. Consider the iterates $\{\boldsymbol{x}_k\}$ and $\{\boldsymbol{y}_k\}$ generated by the method \cref{eq:random_EG}. Then, for all $k\geq 0$, the following inequality holds:
	\begin{equation}
		\|\boldsymbol{x}_{k} - \boldsymbol{x}^*\|^2 - E_{i_k} \left[\|\boldsymbol{x}_{k+1} - \boldsymbol{x}^*\|^2\right] \geq \frac{\rho^2(1-\rho)^2 \|F^2(\boldsymbol{x}_k)\|_{[\gamma]}^2}{\sum_{i=1}^{n}l_i^{\gamma} (l_{\max} + \rho L)^2 \|F(\boldsymbol{x}_k)\|^2}.
		\label{eq:random_EG_convergence_point}
	\end{equation}
	where $F^2(\boldsymbol{x}) = [F^2_1(\boldsymbol{x}), \dots, F^2_n(\boldsymbol{x})]^T$.
\end{lemma}

\begin{proof}
	Proceeding analogously to the proof of \cref{lem:greedy_EG_convergence_point} and using the looser bound in \cref{cor:Fy} instead of \cref{lem:Fy}, we arrive at
	\begin{equation}
		\|\boldsymbol{x}_{k} - \boldsymbol{x}^*\|^2 - \|\boldsymbol{x}_{k+1} - \boldsymbol{x}^*\|^2 \geq \frac{\rho^2(1-\rho)^2 |F_{i_k}(\boldsymbol{x}_k)|^4}{(l_{\max} + \rho L)^2 \|F(\boldsymbol{x}_k)\|^2}.
		\label{eq:random_EG_convergence_point_component}
	\end{equation}
	Taking expectation with respect to the random variable $i_k$, we have
	$$
	\|\boldsymbol{x}_{k} - \boldsymbol{x}^*\|^2 - E_{i_k} \left[\|\boldsymbol{x}_{k+1} - \boldsymbol{x}^*\|^2\right] \geq \frac{\rho^2(1-\rho)^2}{(l_{\max} + \rho L)^2 \|F(\boldsymbol{x}_k)\|^2} \sum_{i=1}^{n} \frac{l_i^{\gamma}}{\sum_{j=1}^{n} l_j^{\gamma}} |F_{i}(\boldsymbol{x}_k)|^4,
	$$
	which proves inequality~\cref{eq:random_EG_convergence_point}.
\end{proof}

\begin{theorem} \label{thm:random_EG_convergence_F}
	Assume that Assumptions~\ref{assumption:monotone}, \ref{assumption:Lipschitz}, and \ref{assumption:comp_lip} hold. Let $\{\boldsymbol{x}_k\}$ and $\{\boldsymbol{y}_k\}$ be the iterates generated by the method \cref{eq:random_EG}. Then, for any integer $K \geq 0$, the following bound holds:
	\begin{equation}
		\min_{0\leq k \leq K} E_{i_{k-1}} \left[\frac{\|F^2(\boldsymbol{x}_k)\|_{[\gamma]}^2}{\|F(\boldsymbol{x}_k)\|^2}\right] \leq \frac{\sum_{i=1}^{n}l_i^{\gamma} (l_{\max} + \rho L)^2}{\rho^2(1-\rho)^2 (K+1)} \|\boldsymbol{x}_{0} - \boldsymbol{x}^*\|^2.
		\label{eq:random_EG_convergence_func}
	\end{equation}
	Furthermore, the expected residual converges to zero:
	\begin{equation}
		\liminf_{k\rightarrow \infty} E_{i_{k-1}} \left[\|F(\boldsymbol{x}_k)\|\right] = 0.
		\label{eq:random_EG_convergence_func_lim}
	\end{equation}
\end{theorem}

\begin{proof}
	Taking expectation with respect to $i_{k-1}$ on both sides of inequality~\cref{eq:random_EG_convergence_point}, we obtain
	$$
	\begin{aligned}
		&E_{i_{k-1}} \left[\|\boldsymbol{x}_{k} - \boldsymbol{x}^*\|^2\right] - E_{i_k} \left[\|\boldsymbol{x}_{k+1} - \boldsymbol{x}^*\|^2\right] \\
		&\geq \frac{\rho^2(1-\rho)^2}{\sum_{i=1}^{n}l_i^{\gamma} (l_{\max} + \rho L)^2} E_{i_{k-1}} \left[\frac{\|F^2(\boldsymbol{x}_k)\|_{[\gamma]}^2}{\|F(\boldsymbol{x}_k)\|^2}\right].
	\end{aligned}
	$$
	Summing over $k=0$ to $K$ yields the bound in \cref{eq:random_EG_convergence_func}, following a similar argument as in the proof of \cref{thm:greedy_EG_convergence_F}.
	
	Taking the limit as $K \rightarrow \infty$, we obtain
	$$
	\liminf_{k\rightarrow \infty} E_{i_{k-1}} \left[\frac{\|F^2(\boldsymbol{x}_k)\|_{[\gamma]}^2}{\|F(\boldsymbol{x}_k)\|^2}\right] = 0,
	$$
	which directly implies the result in \cref{eq:random_EG_convergence_func_lim}.
\end{proof}

\subsection{Watchdog-Max}

\begin{figure}[t]
	\centering
	\includegraphics[width=0.5\columnwidth]{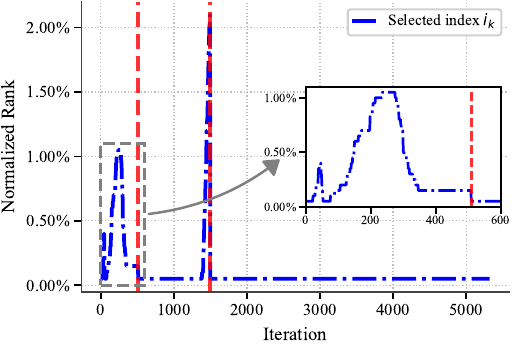}
	\caption{Slow decay of dominant component magnitudes in Watchdog-Max on colon-cancer dataset. The $y$-axis shows the normalized rank, computed as the rank of the component magnitude at the selected index $i_k$ divided by the feature dimension $n$. The red dashed line marks the detection of condition \cref{eq:window_adjust}, triggering adaptive window reset.}
	\label{fig:rankselection}
\end{figure}

Although R-Mini-EG overcomes the need for two full evaluations of the mapping $F$ per iteration in G-Mini-EG, empirical evidence shows that its purely random coordinate selection often requires substantially more iterations, resulting in higher overall computation time. This raises a natural question: how can we design a strategy that retains the computational efficiency of single-evaluation schemes while preserving the fast convergence behavior of greedy approaches?

An important observation is that the magnitude of the mapping component corresponding to the dominant coordinate tends to decay slowly. Specifically, if at $k_0$ the index $i_{k_0} = \arg\max_{1\leq i \leq n} |F_i(\boldsymbol{x}_{k_0})|$ is selected, then over a relatively long window of size $T$, the component $|F_{i_{k_0}}(\boldsymbol{x}_{k})|$ consistently remains among the top-ranked entries for all $k_0 \leq k \leq k_0 + T$.

\begin{remark}
	Although the function value corresponding to $i_{k_0}$ may not remain the exact top-1 entry throughout the window $k_0 \leq k \leq k_0 + T$, it typically stays among the top-ranked coordinates. As a result, the per-iteration efficiency is comparable to that of strictly applying the greedy scheme at every step. The key advantage of this approach is that it requires only a single full evaluation of $F(\boldsymbol{x}_{k_0})$ at the beginning of the window, while avoiding the need to compute $F(\boldsymbol{x}_k)$ at subsequent iteration within the window. This enables the selection of approximately greedy-optimal coordinates with significantly reduced computational cost.
\end{remark}

The window size $T$ can be adaptively adjusted by incorporating a random scheme. Specifically, during each window period $k_0 \leq k \leq k_0 + T$, we compute one additional coordinate value $F_{i_k^r}(\boldsymbol{x}_k)$ at each iteration, where $i_k^r = \mathcal{A}_{\gamma}$.  If at any iteration $k$ it is observed that 
\begin{equation}
	|F_{i_{k_0}}(\boldsymbol{x}_k)| < |F_{i_k^r}(\boldsymbol{x}_k)|,
	\label{eq:window_adjust}
\end{equation}
then iteration $k$ is designated as the new starting point of the next window, and the reference coordinate is updated to $i_k^r$. 

We refer to this adaptive strategy as Watchdog-Max, and its iteration procedure is summarized as follows. Starting from an initial point $\boldsymbol{x}_0 \in \mathbb{R}^n$, we first compute the initial index:
$$
i_0 = \arg\max_{1\leq i \leq n} |F_i(\boldsymbol{x}_0)|.
$$
Then, for each iteration $k \geq 1$, we perform:
\begin{equation}
	\boxed{
		\begin{aligned}
			i_k^r &= \mathcal{A}_{\gamma}, \\
			i_k &= \arg\max_{i_{k-1}, i_k^r} \left\{|F_{i_{k-1}}(\boldsymbol{x}_k)|, |F_{i_k^r}(\boldsymbol{x}_k)|\right\}, \\
			\boldsymbol{y}_k &= \boldsymbol{x}_k - \frac{\rho}{l_{i_k}} F_{i_k}(\boldsymbol{x}_k) e_{i_k}, \\
			\boldsymbol{x}_{k+1} &= P_{\Omega} \left(\boldsymbol{x}_k - \beta_k F(\boldsymbol{y}_k)\right).
		\end{aligned}
	}
	\label{eq:watchdog_max}
\end{equation}
Here, $\rho\in(0,1)$ is a scalar parameter, and $\beta_k$ is the stepsize defined in \cref{eq:beta_comp}.

\begin{remark}
	(i) The Watchdog-Max method performs a full evaluation of the mapping $F$ only once at the initial point to determine the initial greedy-optimal index. In all subsequent iterations, it evaluates only two coordinate values to execute the mini-extragradient step. The inclusion of a randomly sampled coordinate $i_k^r$ enables the algorithm to adaptively update the candidate for the greedy-optimal index. This design substantially reduces the computational cost per iteration and effectively balances the greedy and random selection strategies.
	
	(ii) Empirical evidence indicates that computing the exact greedy-optimal index at every iteration is unnecessary. In fact, the index selection strategy of Watchdog-Max can, in some cases, lead to fewer iterations than G-Mini-EG (see the numerical results in \cref{sec:experiment}).
\end{remark}

According to the index selection rule in \cref{eq:watchdog_max}, each iteration satisfies
\begin{equation}
	|F_{i_k}(\boldsymbol{x}_k)| \geq |F_{i_k^r}(\boldsymbol{x}_k)|,
	\label{eq:wm_over_reg}
\end{equation}
where $i_k^r = \mathcal{A}_{\gamma}$.
Substituting \eqref{eq:wm_over_reg} into \cref{eq:random_EG_convergence_point_component} immediately yields the following result.

\begin{corollary}
	Assume that Assumptions~\ref{assumption:monotone}, \ref{assumption:Lipschitz}, and \ref{assumption:comp_lip} hold. Then, the convergence guarantees established in \cref{thm:random_EG_convergence_F} also apply to the Watchdog-Max method.
\end{corollary}

With these results, we fully resolve the two challenges posed in the beginning. Although R-Mini-EG reduces the per-extragradient-step cost from a full mapping evaluation to essentially a single coordinate evaluation from a theoretical standpoint, the Watchdog-Max strategy further enhances practical efficiency by requiring only one additional coordinate evaluation per iteration.

\begin{table}[t]
	\centering
	\caption{Average performance on three real-world datasets. Best results are shown in \textbf{bold}. The last column reports the \textit{speedup} of Watchdog-Max over EG, calculated as the ratio of their Tcpu values.}
	\label{tab:lr_sota}
	\begin{sc}
		\begin{tabular}{lccccc}
			\toprule
			\textbf{Method} & \multicolumn{1}{c}{\textbf{Tcpu}} & \multicolumn{1}{c}{\textbf{Itr}} & \multicolumn{1}{c}{\textbf{NF}} & \multicolumn{1}{c}{$\boldsymbol{\|F^*\|}$} & \textit{Speedup} \\
			\midrule
			\multicolumn{6}{c}{\textbf{{Colon-cancer}} \cite{alon1999broad}} \\
			\midrule
			\textbf{EG}    & 2.54  & 15606.50  & 31213.00  & {9.87\text{e-}09} &  \multirow{3}[2]{*}{\textit{3.59$\times$}} \\
			\textbf{G-Mini-EG} & 0.87  & \textbf{5989.40}  & 11978.80  & 9.98\text{e-}09 &  \\
			\textbf{Watchdog-Max} & \textbf{0.71}  & 6137.90  & \textbf{6146.44}  & 9.99\text{e-}09 &  \\
			\midrule
			\multicolumn{6}{c}{\textbf{{Duke breast-cancer}} \cite{wen2014robust}} \\
			\midrule
			\textbf{EG}    & 34.44  & 130408.60  & 260817.20  & 9.94\text{e-}09 & \multirow{3}[2]{*}{{\textit{9.25$\times$}} } \\
			\textbf{G-Mini-EG} & 4.49  & \textbf{20225.00}  & 40450.00  & 9.91\text{e-}09 &  \\
			\textbf{Watchdog-Max} & \textbf{3.72}  & 20584.10  & \textbf{20593.17}  & {9.89\text{e-}09} &  \\
			\midrule
			\multicolumn{6}{c}{\textbf{{Leukemia}} \cite{golub1999molecular}} \\
			\midrule
			\textbf{EG}    & 23.61  & 91100.40  & 182200.80  & 9.83\text{e-}09 & \multirow{3}[2]{*}{{\textit{5.90$\times$}} } \\
			\textbf{G-Mini-EG} & 5.00  & \textbf{24479.40}  & 48958.80  & {9.69\text{e-}09} &  \\
			\textbf{Watchdog-Max} & \textbf{4.01}  & 25384.10  & \textbf{25394.62}  & 9.71\text{e-}09 &  \\
			\bottomrule
		\end{tabular}
	\end{sc}
\end{table}

\section{Numerical Experiments} \label{sec:experiment}

In this section, we evaluate the meta-level improvements achieved by the proposed Mini-EG variants over the vanilla EG method on two representative tasks: regularized decentralized logistic regression and compressed sensing. The comparison is carried out against the two-time-scale EG baseline in \cref{eq:eg}. Additional comparisons with recent baseline methods are provided in \cref{tab:cs_sota} of \cref{appendix:cs_sota}.

To ensure a fair comparison, the same stepsize strategy is applied to both the EG and Mini-EG variants. For all algorithms, we set $\rho = 0.999$ to allow a relatively large stepsize. The sensitivity of $\rho$ for both EG and Mini-EG is further examined in \cref{appendix:sensitivity_rho}. In addition, we set $\gamma = 0$ for the Watchdog-Max method. All algorithms terminate when $\|F(\boldsymbol{y}_k)\| \leq 10^{-8}$ or when the iteration count exceeds $500{,}000$.

Each experiment is repeated over 100 independent trials. We report the average number of iterations (\texttt{Itr}), function evaluations (\texttt{NF})\footnote{For Mini-EG methods, each component function evaluation counts as $1/n$ toward \texttt{NF}.}, CPU time in seconds (\texttt{Tcpu}), and the final residual $\|F(\boldsymbol{x})\|$ ($\|F^*\|$).

All experiments were performed in Python 3.9.12 on a Windows system, using a laptop equipped with an Intel i7-12700H processor (2.30 GHz) and 16 GB RAM.

\subsection{Regularized Decentralized Logistic Regression}

We first consider regularized decentralized logistic regression \cite{jian2022family}, which arises in regimes where the feature dimension $n$ far exceeds the sample size $N$. The problem is formulated as:
\begin{equation}
	\min _{\boldsymbol{x} \in \mathbb{R}^n} f(\boldsymbol{x})=\frac{1}{N} \sum_{i=1}^N \log \left(1+\exp \left(-b_i \boldsymbol{a}_i^T \boldsymbol{x}\right)\right)+\frac{\tau}{2}\|\boldsymbol{x}\|^2,
	\label{eq:lr_v0}
\end{equation}
where $\tau > 0$ is a regularization parameter, and $\left(\boldsymbol{a}_i, b_i\right) \in \mathbb{R}^n \times \{-1,1\}$ for $i = 1, \dots, N$ are data samples drawn from a given dataset or distribution. In all experiments, we fix the regularization parameter as $\tau = 0.1$.

\begin{figure}[t]
	\centering
	\begin{subfigure}[b]{0.32\linewidth}
		\centering
		\includegraphics[width=\linewidth]{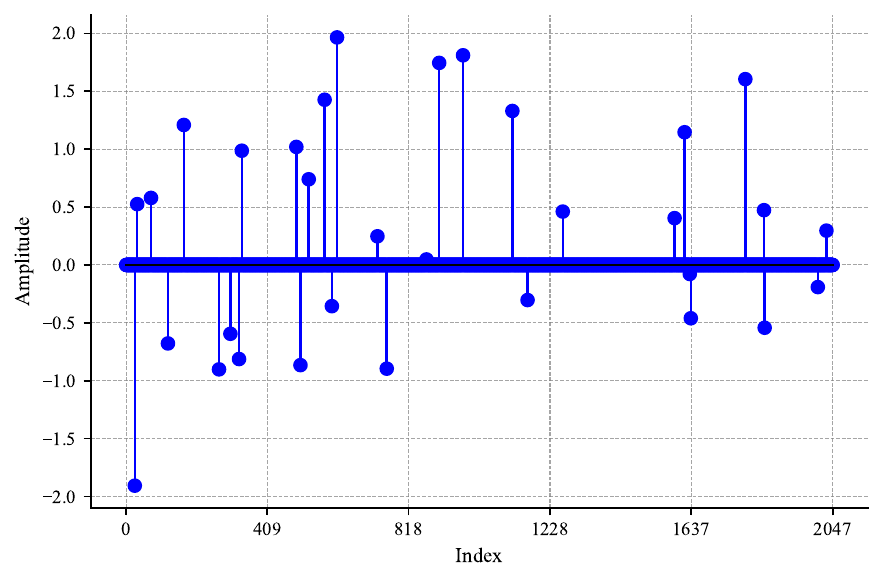}
		\caption{}
		\label{fig:ground_truth}
	\end{subfigure}
	\hfill
	\begin{subfigure}[b]{0.32\linewidth}
		\centering
		\includegraphics[width=\linewidth]{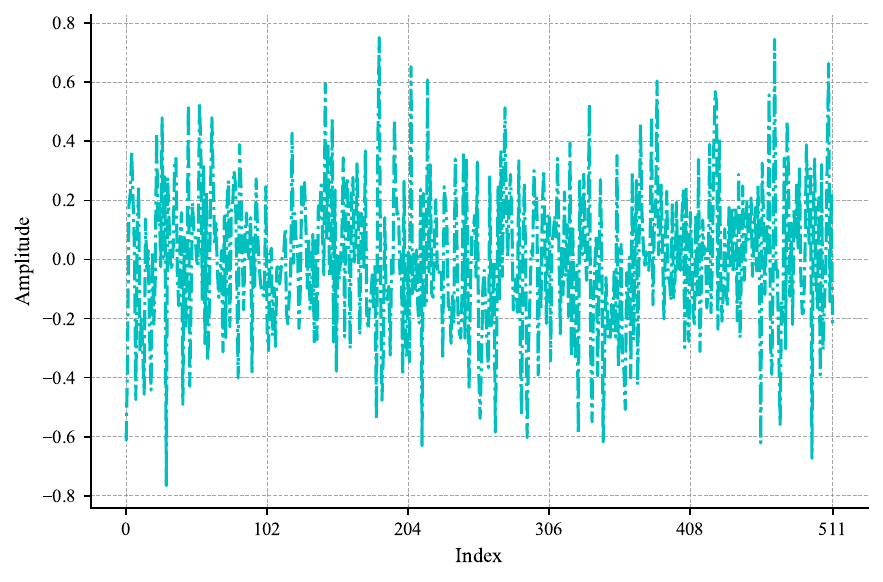}
		\caption{}
		\label{fig:observation_b}
	\end{subfigure}
	\hfill
	\begin{subfigure}[b]{0.32\linewidth}
		\centering
		\includegraphics[width=\linewidth]{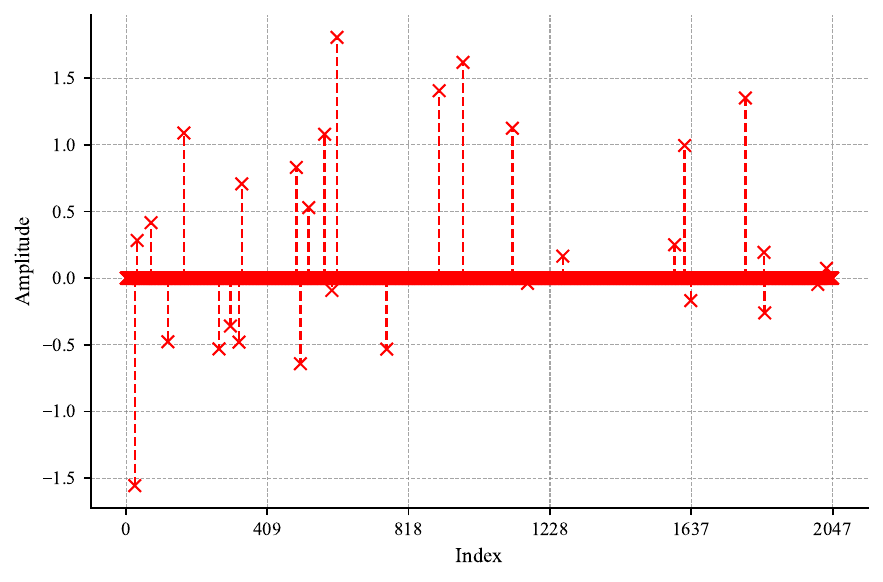}
		\caption{}
		\label{fig:x_recovered}
	\end{subfigure}
	\caption{Signal recovery visualization using Watchdog-Max with $n=2048$, $N=512$, $K=32$, and SNR = 20 dB. (a) Original signal; (b) Observation; (c) Reconstruction.}
	\label{fig:signal_recovery_all}
\end{figure}

It is well known that the objective function $f$ in \cref{eq:lr_v0} is smooth and strongly convex, owing to the presence of both the smooth logistic loss and the $\ell_2$ regularization term.
Therefore, solving \cref{eq:lr_v0} is equivalent to finding the root of the following system of nonlinear monotone equations:
$$
F(\boldsymbol{x})=\nabla f(\boldsymbol{x})=\frac{1}{N} \sum_{i=1}^N \frac{-b_i \exp \left(-b_i \boldsymbol{a}_i^T \boldsymbol{x}\right)}{1+\exp \left(-b_i \boldsymbol{a}_i^T \boldsymbol{x}\right)} \boldsymbol{a}_i+\tau \boldsymbol{x}=\boldsymbol{0}.
$$
Let $\boldsymbol{A} = \left[\boldsymbol{a}_1,\dots,\boldsymbol{a}_N\right]$ and define $\boldsymbol{H} = \boldsymbol{A}\boldsymbol{A}^T$. Through standard analysis, it can be verified that the mapping $F$ satisfies both Assumptions~\ref{assumption:Lipschitz} and \ref{assumption:comp_lip}. Specifically, the global Lipschitz constant of $F$ is given by
\begin{equation}
	L = \frac{1}{4N} \lambda_1\left(\boldsymbol{H}\right) + \tau,
	\label{eq:lr_L}
\end{equation}
where $\lambda_1(\boldsymbol{H})$ denotes the largest eigenvalue of $\boldsymbol{H}$.

In contrast, the componentwise Lipschitz constants take the form
\begin{equation}
	l_i = \frac{1}{4N} h_{ii} + \tau,\quad i = 1,\dots, n,
	\label{eq:lr_li}
\end{equation}
where $h_{ii}$ is the $(i,i)$-th diagonal entry of $\boldsymbol{H}$.

\begin{remark}
	Comparing \cref{eq:lr_L} and \cref{eq:lr_li}, computing the constant $L$ requires estimating the dominant eigenvalue of $\boldsymbol{H}$, which is computationally expensive when $n$ is large. In contrast, computing each $l_i$ only needs a diagonal entry of $\boldsymbol{H}$, rendering its computation considerably more efficient.
\end{remark}

Our experiments use three real-world datasets from the LIBSVM repository\footnote{\url{https://www.csie.ntu.edu.tw/~cjlin/libsvmtools/datasets/}} \cite{chang2011libsvm}: colon-cancer ($N = 62$, $n = 2{,}000$), duke breast-cancer ($N = 44$, $n = 7{,}129$), and leukemia ($N = 38$, $n = 7{,}129$).

\Cref{fig:rankselection} illustrates the slow decay of dominant component magnitudes in the Watchdog-Max method on the colon-cancer dataset. Throughout the iterations, the component selected at index $i_k$ consistently ranks within the top 2\%, with only two window resets triggered. In the final window, it persistently coincides with the top-1 component. This observation indicates that performing full evaluations at every extragradient step is unnecessary for identifying the dominant component.

\cref{tab:lr_sota} demonstrates that the Mini-EG variants outperform the standard EG method, with Watchdog-Max achieving a $9.25\times$ speedup on duke breast-cancer data.
Notably, Mini-EG methods reduce both the per-iteration cost and the total number of iterations compared with EG. Among them, G-Mini-EG attains the lowest iteration count, whereas Watchdog-Max requires slightly more iterations but updates only two components per extragradient step, which leads to substantially lower \texttt{NF} and \texttt{Tcpu}. These results indicate that full mapping evaluations at every extragradient step are not necessary, and that updating only a single coordinate per extragradient step already provides significant computational gains.

\begin{figure}[t]
	\centering
	\begin{subfigure}[b]{0.32\linewidth}
		\centering
		\includegraphics[width=\textwidth]{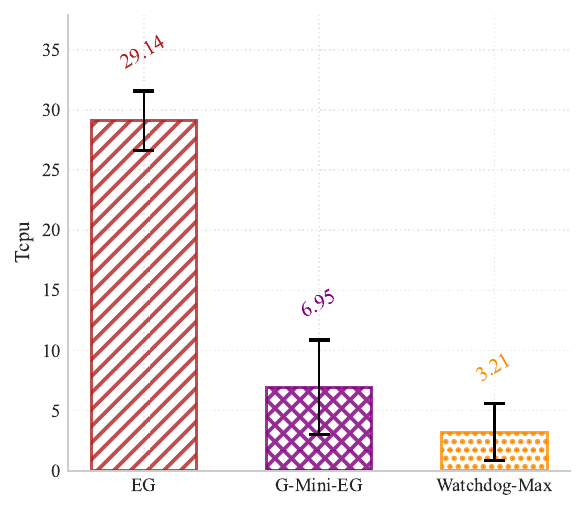}
		\caption{}
		\label{fig:Time_comparison}
	\end{subfigure}
	\hfill
	\begin{subfigure}[b]{0.32\linewidth}
		\centering
		\includegraphics[width=\textwidth]{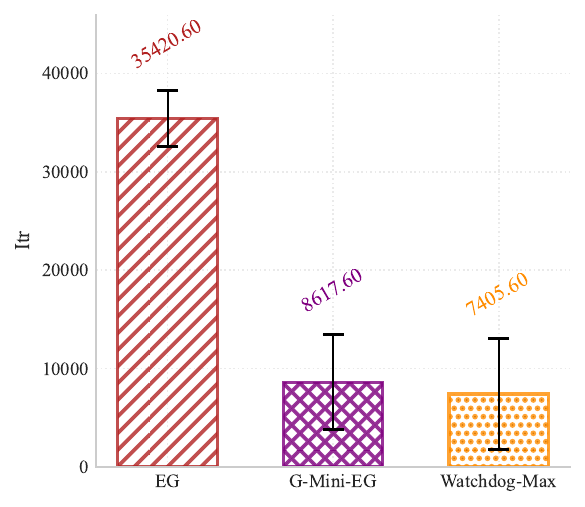}
		\caption{}
		\label{fig:Iter_comparison}
	\end{subfigure}
	\hfill
	\begin{subfigure}[b]{0.32\linewidth}
		\centering
		\includegraphics[width=\textwidth]{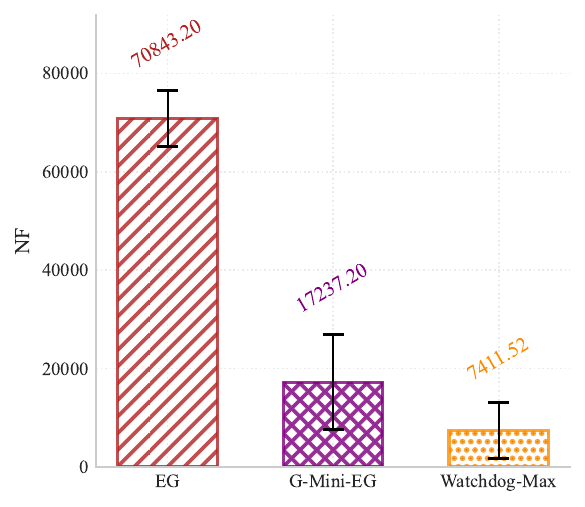}
		\caption{}
		\label{fig:NF_comparison}
	\end{subfigure}
	\caption{Performance comparison of different methods is shown, in terms of (a) \texttt{Tcpu}, (b) \texttt{Itr}, and (c) \texttt{NF}.}
	\label{fig:cs_bar_plot}
\end{figure}

\subsection{Compressed Sensing} \label{sec:cs}

We consider the LASSO problem:
\begin{equation}
	\min_{\boldsymbol{x} \in \mathbb{R}^n} f(\boldsymbol{x}) = \frac{1}{2}\|\boldsymbol{A} \boldsymbol{x} - \boldsymbol{b}\|_2^2 + \tau \|\boldsymbol{x}\|_1,
	\label{eq:lasso}
\end{equation}
where $\boldsymbol{A} \in \mathbb{R}^{N \times n}$ ($N \leq n$) is the sensing matrix, $\boldsymbol{b} \in \mathbb{R}^N$ is the observation vector, and $\tau > 0$ is a regularization parameter promoting sparsity.
Following \cite{kim2008interior}, we set $\tau = 0.1\|\boldsymbol{A}^T\boldsymbol{b}\|_{\infty}$.

\begin{table}[t]
	\centering
	\caption{Average performance comparison of Mini-EG and recent EG-type methods in compressed sensing under different configurations. The best results are highlighted in \textbf{bold}, and the \colorbox{gray!10}{\textit{speedup}} of Watchdog-Max over EG is also reported.}
	\label{tab:cs_sota}
	\resizebox{1\linewidth}{!}
	{
		\begin{small}
			\begin{sc}
				\begin{tabular}{lccc}
					\toprule
					\textbf{Method} & $K = 8$ & $K = 16$ & $K = 32$ \\
					\cmidrule(lr){2-2} \cmidrule(lr){3-3} \cmidrule(lr){4-4}
					& \scriptsize \textbf{TCPU/Itr/NF/$\boldsymbol{\|F^*\|}$} & \scriptsize \textbf{TCPU/Itr/NF/$\boldsymbol{\|F^*\|}$} & \scriptsize \textbf{TCPU/Itr/NF/$\boldsymbol{\|F^*\|}$} \\
					\midrule
					\multicolumn{4}{c}{\textbf{$N = 512$}} \\
					\cmidrule(lr){2-4}
					\textbf{EG} & \scriptsize 6.24/24830.80/49661.60/1.00\text{e-}08 & \scriptsize 6.94/28310.80/56621.60/1.00\text{e-}08 & \scriptsize 8.32/35517.50/71035.00/1.00\text{e-}08 \\
					\textbf{PEG} & \scriptsize 4.03/24851.10/24851.10/1.00\text{e-}08 & \scriptsize 4.23/28331.80/28331.80/1.00\text{e-}08 & \scriptsize 5.07/35538.60/35538.60/1.00\text{e-}08 \\
					\textbf{EAG-C} & \scriptsize 11.44/50000.00/100000.00/4.96\text{e-}04 & \scriptsize 11.51/50000.00/100000.00/6.64\text{e-}04 & \scriptsize 11.41/50000.00/100000.00/9.27\text{e-}04 \\
					\textbf{FEG} & \scriptsize 11.55/50000.00/100000.00/1.70\text{e-}01 & \scriptsize 11.75/50000.00/100000.00/2.29\text{e-}01 & \scriptsize 11.64/50000.00/100000.00/3.21\text{e-}01 \\
					\textbf{G-Mini-EG} & \scriptsize 0.92/4707.50/9415.00/1.00\text{e-}08 & \scriptsize 1.55/\textbf{7007.80}/14015.60/1.00\text{e-}08 & \scriptsize 1.45/7508.20/15016.40/1.00\text{e-}08 \\
					\textbf{Watchdog-Max} & \scriptsize \textbf{0.46}/\textbf{4020.60}/\textbf{4024.36}/1.00\text{e-}08 & \scriptsize \textbf{0.83}/7595.90/\textbf{7601.01}/1.00\text{e-}08 & \scriptsize \textbf{0.73}/\textbf{6347.90}/\textbf{6352.60}/1.00\text{e-}08 \\
					\cmidrule(lr){2-2} \cmidrule(lr){3-3} \cmidrule(lr){4-4}
					\rowcolor{gray!10} \textit{Speedup} & \scriptsize \textit{13.51$\times$} & \scriptsize \textit{8.35$\times$} & \scriptsize \textit{11.36$\times$} \\
					\midrule
					\multicolumn{4}{c}{\textbf{$N = 768$}} \\
					\cmidrule(lr){2-4}
					\textbf{EG} & \scriptsize 5.09/18946.60/37893.20/1.00\text{e-}08 & \scriptsize 5.44/20763.70/41527.40/1.00\text{e-}08 & \scriptsize 5.85/23137.00/46274.00/1.00\text{e-}08 \\
					\textbf{PEG} & \scriptsize 3.36/18966.90/18966.90/9.99\text{e-}09 & \scriptsize 3.54/20784.60/20784.60/1.00\text{e-}08 & \scriptsize 3.81/23158.10/23158.10/1.00\text{e-}08 \\
					\textbf{EAG-C} & \scriptsize 11.32/50000.00/100000.00/4.29\text{e-}04 & \scriptsize 11.29/50000.00/100000.00/7.33\text{e-}04 & \scriptsize 11.20/50000.00/100000.00/9.20\text{e-}04 \\
					\textbf{FEG} & \scriptsize 11.49/50000.00/100000.00/1.13\text{e-}01 & \scriptsize 11.55/50000.00/100000.00/1.94\text{e-}01 & \scriptsize 11.40/50000.00/100000.00/2.44\text{e-}01 \\
					\textbf{G-Mini-EG} & \scriptsize 0.70/\textbf{3626.50}/7253.00/1.00\text{e-}08 & \scriptsize 0.74/3957.60/7915.20/1.00\text{e-}08 & \scriptsize 1.34/6860.20/13720.40/1.00\text{e-}08 \\
					\textbf{Watchdog-Max} & \scriptsize \textbf{0.62}/5134.80/\textbf{5138.71}/1.00\text{e-}08 & \scriptsize \textbf{0.45}/\textbf{3486.90}/\textbf{3490.00}/1.00\text{e-}08 & \scriptsize \textbf{0.70}/\textbf{5939.70}/\textbf{5943.80}/1.00\text{e-}08 \\
					\cmidrule(lr){2-2} \cmidrule(lr){3-3} \cmidrule(lr){4-4}
					\rowcolor{gray!10} \textit{Speedup} & \scriptsize \textit{8.24$\times$} & \scriptsize \textit{12.19$\times$} & \scriptsize \textit{8.31$\times$} \\
					\midrule
					\multicolumn{4}{c}{\textbf{$N = 1024$}} \\
					\cmidrule(lr){2-4}
					\textbf{EG} & \scriptsize 4.52/15660.20/31320.40/1.00\text{e-}08 & \scriptsize 4.68/16771.10/33542.20/1.00\text{e-}08 & \scriptsize 5.08/18511.30/37022.60/1.00\text{e-}08 \\
					\textbf{PEG} & \scriptsize 3.12/15680.60/15680.60/9.99\text{e-}09 & \scriptsize 3.10/16791.60/16791.60/1.00\text{e-}08 & \scriptsize 3.36/18532.40/18532.40/9.99\text{e-}09 \\
					\textbf{EAG-C} & \scriptsize 11.25/50000.00/100000.00/5.06\text{e-}04 & \scriptsize 11.47/50000.00/100000.00/6.72\text{e-}04 & \scriptsize 11.69/50000.00/100000.00/9.65\text{e-}04 \\
					\textbf{FEG} & \scriptsize 11.46/50000.00/100000.00/1.13\text{e-}01 & \scriptsize 11.67/50000.00/100000.00/1.49\text{e-}01 & \scriptsize 11.77/50000.00/100000.00/2.14\text{e-}01 \\
					\textbf{G-Mini-EG} & \scriptsize 1.03/\textbf{5481.90}/10963.80/1.00\text{e-}08 & \scriptsize 1.12/\textbf{5556.30}/11112.60/1.00\text{e-}08 & \scriptsize 1.31/7032.70/14065.40/1.00\text{e-}08 \\
					\textbf{Watchdog-Max} & \scriptsize \textbf{0.63}/5634.40/\textbf{5638.75}/1.00\text{e-}08 & \scriptsize \textbf{0.73}/6082.20/\textbf{6086.77}/1.00\text{e-}08 & \scriptsize \textbf{0.77}/\textbf{6428.20}/\textbf{6432.74}/1.00\text{e-}08 \\
					\cmidrule(lr){2-2} \cmidrule(lr){3-3} \cmidrule(lr){4-4}
					\rowcolor{gray!10} \textit{Speedup} & \scriptsize \textit{7.14$\times$} & \scriptsize \textit{6.40$\times$} & \scriptsize \textit{6.58$\times$} \\
					\bottomrule
				\end{tabular}
			\end{sc}
		\end{small}
	}
\end{table}

By decomposing $\boldsymbol{x}$ into its positive and negative parts, i.e., $\boldsymbol{x} = \boldsymbol{u} - \boldsymbol{v}$, problem \cref{eq:lasso} can be reformulated as a system of nonlinear monotone equations \cite{figueiredo2008gradient}:
\begin{equation}
	F(\boldsymbol{z}) = \min \{\boldsymbol{z}, \boldsymbol{H} \boldsymbol{z}+\boldsymbol{c}\}=\boldsymbol{0},
	\label{eq:F_cs}
\end{equation}
where 
$$
\boldsymbol{z}=\left[\begin{array}{l}\boldsymbol{u} \\ \boldsymbol{v}\end{array}\right], \quad \boldsymbol{H}=\left[\begin{array}{cc}
	\boldsymbol{A}^T \boldsymbol{A} & -\boldsymbol{A}^T \boldsymbol{A} \\
	-\boldsymbol{A}^T \boldsymbol{A} & \boldsymbol{A}^T \boldsymbol{A}
\end{array}\right],
$$
and 
$$
\boldsymbol{c}=\tau \boldsymbol{1}_{2 n}+\left[\begin{array}{c}
	-\boldsymbol{A}^T\boldsymbol{b} \\
	\boldsymbol{A}^T\boldsymbol{b}
\end{array}\right],
$$
with $\boldsymbol{1}_{2n}$ denoting the all-one vector in $\mathbb{R}^{2n}$.
The operator $\min\left\{\cdot\right\}$ is applied componentwise.

Following \cite[Lemma 3]{pang1986inexact}, it can be shown that the mapping $F$ in \cref{eq:F_cs} has a global Lipschitz constant
\begin{equation}
	L = \sqrt{n} \max\left\{\lambda_1(\boldsymbol{H}),1\right\},
	\label{eq:cs_L}
\end{equation}
and componentwise Lipschitz constants
\begin{equation}
	l_i = \max\left\{h_{ii}, 1\right\},\quad i = 1,\dots, 2n.
	\label{eq:cs_li}
\end{equation}

\begin{remark}
	From \cref{eq:cs_L} and \cref{eq:cs_li}, we observe that $l_i$ is easier to compute and typically smaller than $L$. Furthermore, leveraging the structure of $\boldsymbol{H}$, it suffices to compute only the first $n$ diagonal entries $h_{ii}$, i.e., the diagonal elements of $\boldsymbol{A}^T\boldsymbol{A}$.
\end{remark}

We generate synthetic datasets by randomly sampling sparse signals with sparsity level $K$. The sensing matrix $\boldsymbol{A}$ is drawn from a standard Gaussian distribution, and the observation signal-to-noise ratio (SNR) is set to 20 dB.

Under the setting $n=2048$, $N=512$, and $K=32$ with a zero initial point, \cref{fig:signal_recovery_all} illustrates the signal recovery performance of Watchdog-Max. \cref{fig:cs_bar_plot} compares methods via bar plots (mean$\pm$std) for \texttt{Tcpu}, \texttt{Itr}, and \texttt{NF}. Watchdog-Max outperforms all baselines, achieving over $9\times$ and $2\times$ speedup in \texttt{Tcpu} versus EG and G-Mini-EG, respectively. Notably, it requires fewer iterations than G-Mini-EG, indicating that the greedy index-selection scheme in G-Mini-EG admits further improvement.

\begin{figure}[t]
	\centering
	\begin{subfigure}[b]{0.32\linewidth}
		\centering
		\includegraphics[width=\textwidth]{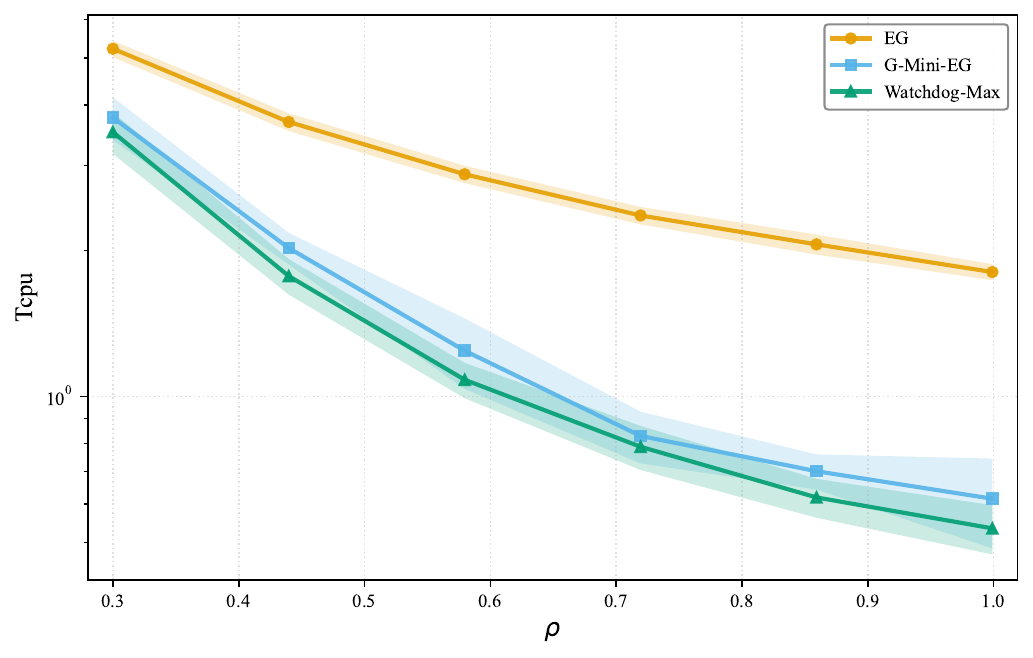}
		\caption{}
		\label{fig:cancer_Time_vs_rho}
	\end{subfigure}
	\hfill
	\begin{subfigure}[b]{0.32\linewidth}
		\centering
		\includegraphics[width=\textwidth]{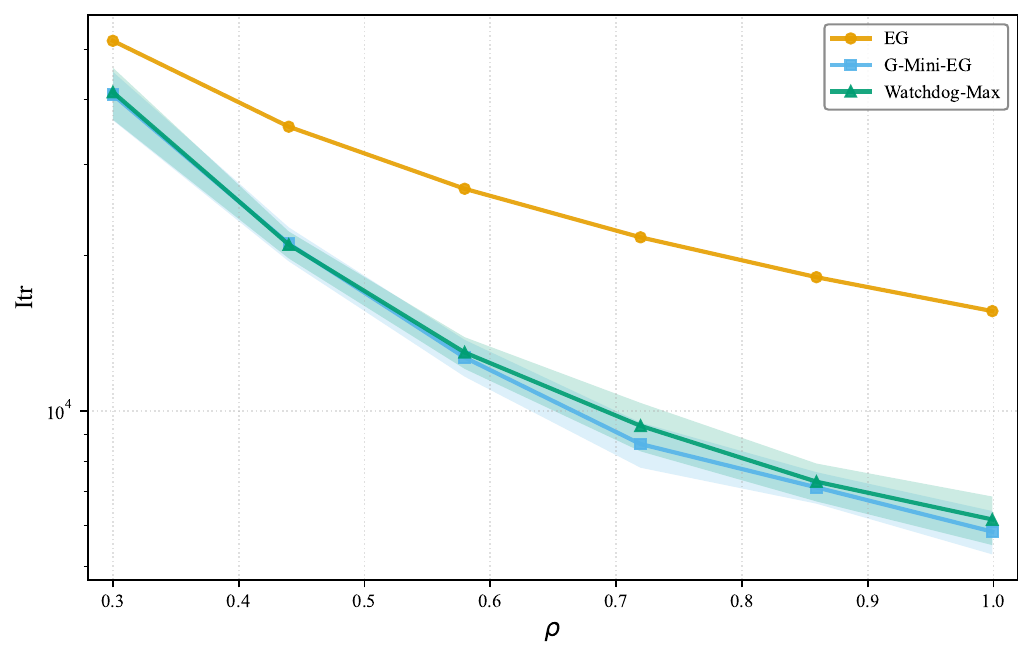}
		\caption{}
		\label{fig:cancer_Iter_vs_rho}
	\end{subfigure}
	\hfill
	\begin{subfigure}[b]{0.32\linewidth}
		\centering
		\includegraphics[width=\textwidth]{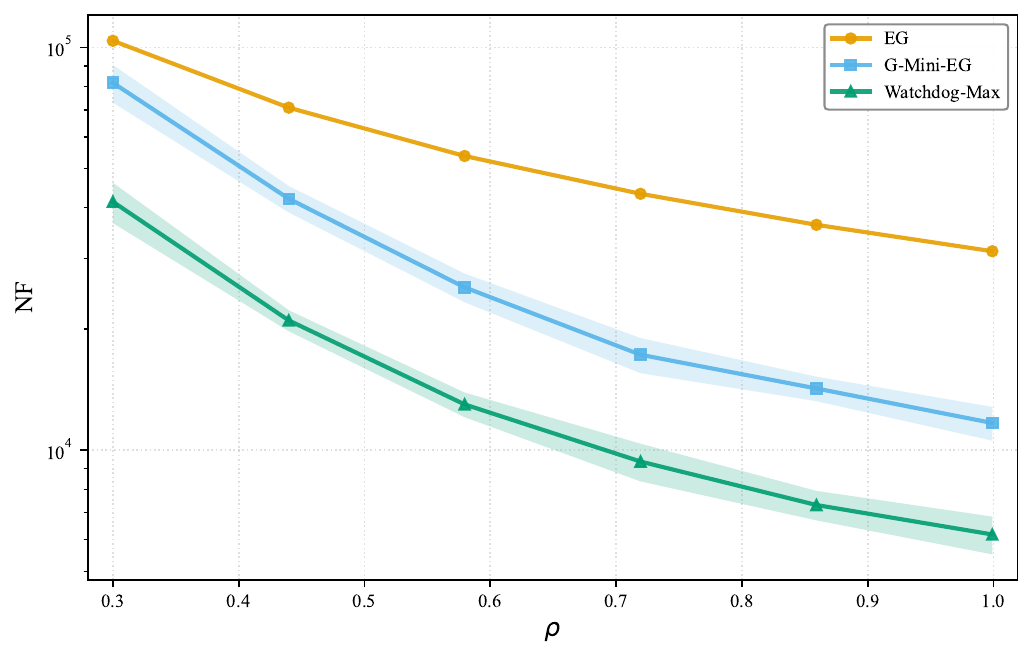}
		\caption{}
		\label{fig:cancer_NF_vs_rho}
	\end{subfigure}
	\begin{subfigure}[b]{0.32\linewidth}
		\centering
		\includegraphics[width=\textwidth]{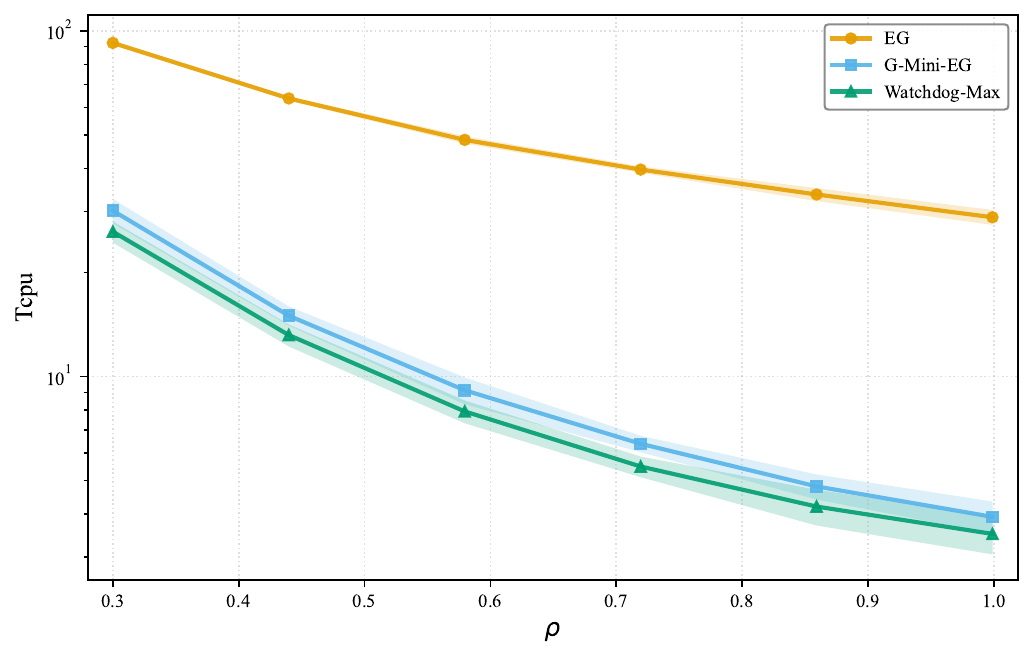}
		\caption{}
		\label{fig:duke_Time_vs_rho}
	\end{subfigure}
	\hfill
	\begin{subfigure}[b]{0.32\linewidth}
		\centering
		\includegraphics[width=\textwidth]{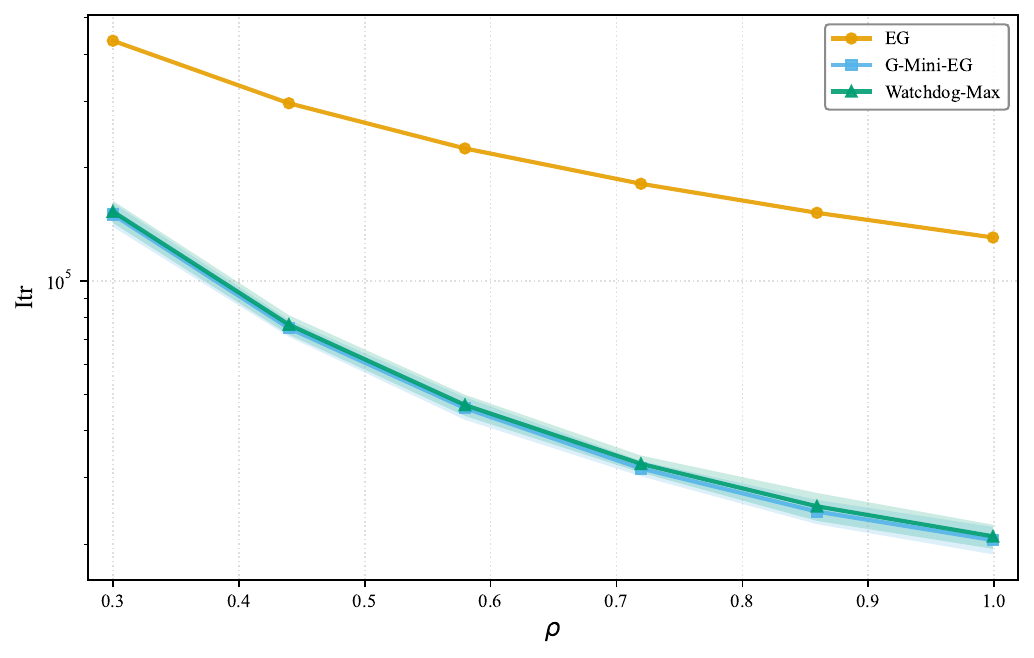}
		\caption{}
		\label{fig:duke_Iter_vs_rho}
	\end{subfigure}
	\hfill
	\begin{subfigure}[b]{0.32\linewidth}
		\centering
		\includegraphics[width=\textwidth]{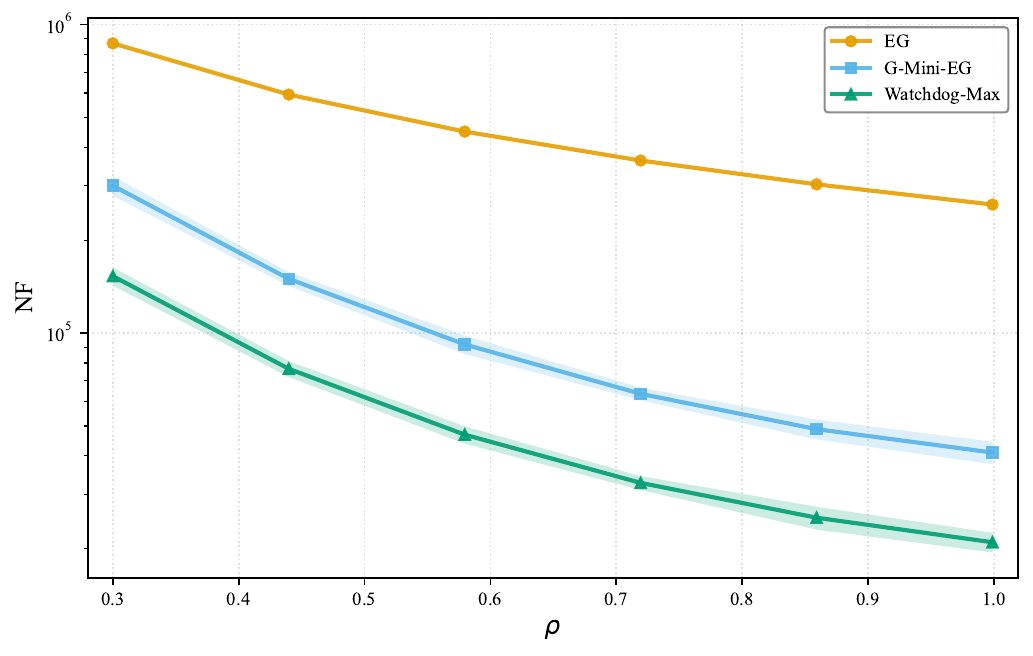}
		\caption{}
		\label{fig:duke_NF_vs_rho}
	\end{subfigure}
	\begin{subfigure}[b]{0.32\linewidth}
		\centering
		\includegraphics[width=\textwidth]{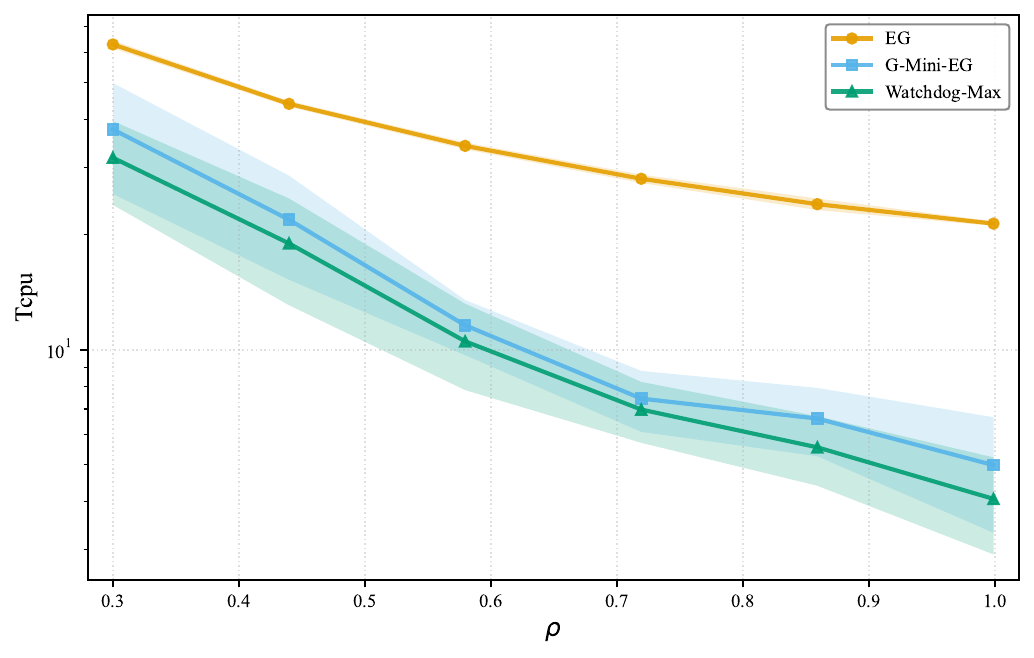}
		\caption{}
		\label{fig:leu_Time_vs_rho}
	\end{subfigure}
	\hfill
	\begin{subfigure}[b]{0.32\linewidth}
		\centering
		\includegraphics[width=\textwidth]{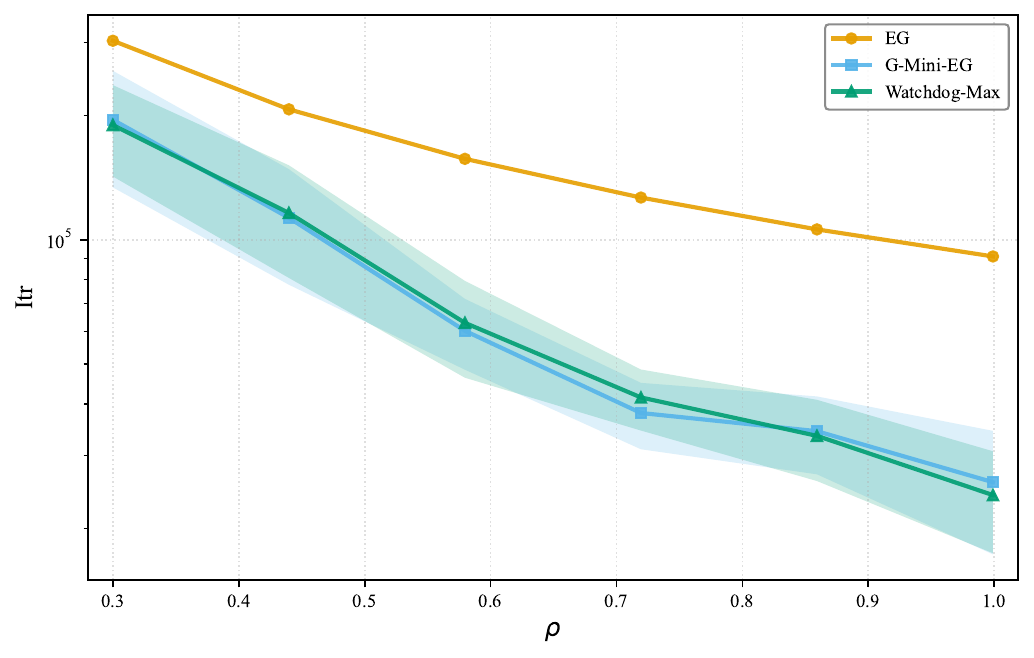}
		\caption{}
		\label{fig:leu_Iter_vs_rho}
	\end{subfigure}
	\hfill
	\begin{subfigure}[b]{0.32\linewidth}
		\centering
		\includegraphics[width=\textwidth]{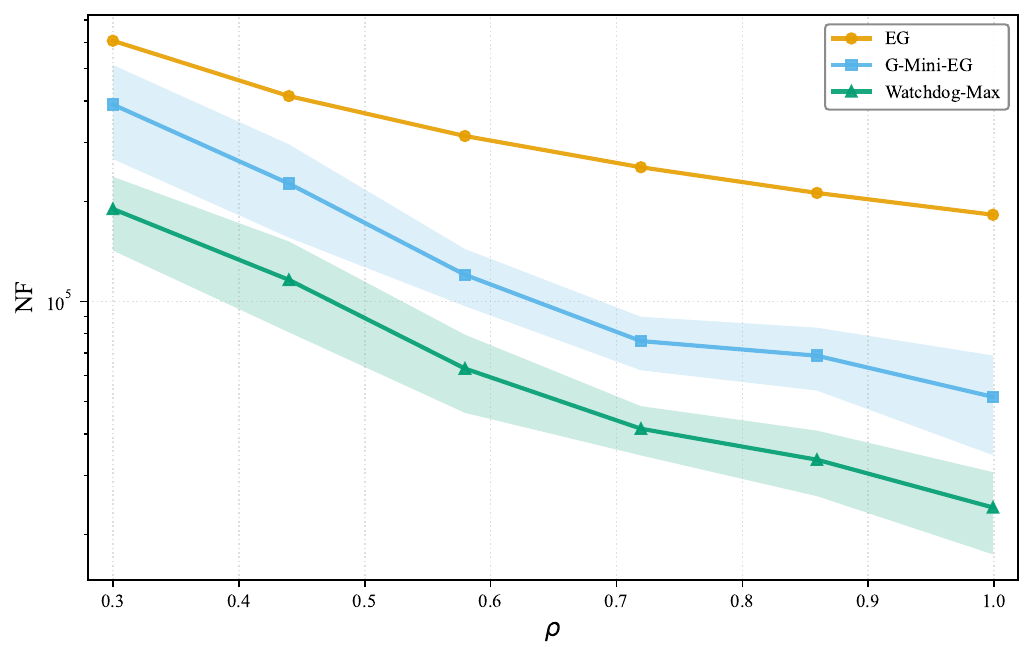}
		\caption{}
		\label{fig:leu_NF_vs_rho}
	\end{subfigure}
	\caption{Sensitivity of $\rho$ for EG and Mini-EG on regularized decentralized logistic regression, measured by \texttt{Tcpu}, \texttt{Itr}, and \texttt{NF}. Curves show average results, and shaded areas indicate one standard deviation. Datasets: colon-cancer (a-c), duke breast-cancer (d-f), leukemia (g-i).}
	\label{fig:LR_rho}
\end{figure}

\section{Conclusion and Future Work} \label{sec:conclusion}

This work introduces a family of Mini-EG methods for solving monotone nonlinear equations. By updating only the dominant coordinate and employing componentwise Lipschitz constants, the proposed G-Mini-EG method achieves more efficient iterations together with stepsizes that are substantially easier to compute. To further reduce computational overhead, the Watchdog-Max strategy incorporates randomized sampling so that only two coordinates are tracked at each extragradient step, while still maintaining both convergence guarantees and computational efficiency. Theoretical results are established for all proposed variants. Notably, in a range of commonly encountered problem settings, the Greedy Mini-EG method admits sharper convergence bounds than those available for the standard EG method. Empirical evaluations on decentralized logistic regression and compressed sensing further demonstrate runtime improvements of more than ${13\times}$ over the classical EG algorithm.

\textbf{Future work.}
These findings significantly advance the applicability of EG-type methods to problems in optimization and scientific computing. Future work will investigate Mini-Batch extensions by selecting a subset of coordinates at each step, and explore integrating the mini-update strategy into both stages of the algorithm, paving the way for a potential Double-Mini-EG framework. Additionally, we plan to analyze the last-iterate convergence rate to address a gap in the convergence theory of Mini-EG methods.

\appendix

\section{Additional Comparisons with Recent Baselines} \label{appendix:cs_sota}

In this section, we present additional comparisons between the proposed Mini-EG algorithms and several recent baseline methods, including PEG \cite{popov1980modification}, EAG \cite{yoon2021accelerated}, and FEG \cite{lee2021fast}, in the context of compressed sensing.
We vary $N \in \{512, 768, 1024\}$ and $K \in \{8, 16, 32\}$ under $n = 2048$ and SNR = 20~dB. The maximum number of iterations is set to $50{,}000$. The same stepsize is adopted for PEG, EG, and Mini-EG. For EAG with constant stepsize (EAG-C), the stepsize is tuned to the largest value within its convergence range. All parameters of FEG follow the default settings recommended in \cite{lee2021fast}.

\Cref{tab:cs_sota} reports the averaged results over all configurations. As observed, the Mini-EG variants consistently outperform existing EG-type methods across all performance metrics. In particular, Watchdog-Max achieves the best results in terms of \texttt{Tcpu} in every case, with a maximum speedup exceeding $13\times$ relative to EG. It also demonstrates a clear improvement over G-Mini-EG, which is mainly due to the fact that each extragradient step in Watchdog-Max updates only two coordinates, thereby considerably reducing the \texttt{NF} while maintaining comparable iteration counts. Moreover, both EAG-C and FEG fail to reach the prescribed accuracy within the iteration limit, converging only to low-accuracy solutions.

\section{Sensitivity Analysis of the Parameter $\rho$} \label{appendix:sensitivity_rho}

In this section, we analyze the sensitivity of the parameter $\rho$ and its impact on the performance of the EG and Mini-EG methods. As shown in \cref{fig:LR_rho}, for different values of $\rho \in (0,1)$, both G-Mini-EG and Watchdog-Max consistently outperform EG in terms of \texttt{Tcpu}, \texttt{Itr}, and \texttt{NF}. Moreover, as $\rho$ increases, the acceleration achieved by Mini-EG relative to EG becomes more pronounced. These observations suggest that selecting $\rho$ close to 1 can be beneficial for achieving faster convergence in certain scenarios.\footnote{Although, in theory, setting $\rho$ near 1 does not necessarily yield the tightest convergence rate bound, in practice faster convergence may arise because $l_{i_k}$ remains an overestimate.}

\bibliographystyle{siamplain}
\bibliography{references}

@article{nesterov2012efficiency,
  title={Efficiency of coordinate descent methods on huge-scale optimization problems},
  author={Nesterov, Yu},
  journal={SIAM J. Optim.},
  volume={22},
  number={2},
  pages={341--362},
  year={2012}
}

@article{korpelevich1976extragradient,
  title={The extragradient method for finding saddle points and other problems},
  author={Korpelevich, Galina M},
  journal={Matecon},
  volume={12},
  pages={747--756},
  year={1976}
}

@article{jian2022family,
  title={A family of inertial derivative-free projection methods for constrained nonlinear pseudo-monotone equations with applications},
  author={Jian, Jinbao and Yin, Jianghua and Tang, Chunming and Han, Daolan},
  journal={Comput. Appl. Math.},
  volume={41},
  number={7},
  pages={309},
  year={2022}
}

@inproceedings{goodfellow2014generative,
  title={Generative adversarial nets},
  author={Goodfellow, Ian J and Pouget-Abadie, Jean and Mirza, Mehdi and Xu, Bing and Warde-Farley, David and Ozair, Sherjil and Courville, Aaron and Bengio, Yoshua},
  booktitle = {Proceedings of the 28th International Conference on Neural Information Processing Systems},
  publisher = {Curran Associates, Inc.},
  volume={27},
  year={2014}
}

@article{figueiredo2008gradient,
  title={Gradient projection for sparse reconstruction: Application to compressed sensing and other inverse problems},
  author={Figueiredo, M{\'a}rio AT and Nowak, Robert D and Wright, Stephen J},
  journal={IEEE J. Sel. Top. Signal Process.},
  volume={1},
  number={4},
  pages={586--597},
  year={2008}
}

@inproceedings{wen2014robust,
  title={Robust learning under uncertain test distributions: Relating covariate shift to model misspecification},
  author={Wen, Junfeng and Yu, Chun-Nam and Greiner, Russell},
  booktitle={Proceedings of the 31st International Conference on Machine Learning},
  pages={631--639},
  year={2014},
  organization={PMLR}
}

@article{wang2020generalizing,
  title={Generalizing from a few examples: A survey on few-shot learning},
  author={Wang, Yaqing and Yao, Quanming and Kwok, James T and Ni, Lionel M},
  journal={ACM Comput. Surv.},
  volume={53},
  number={3},
  pages={1--34},
  year={2020}
}

@inproceedings{diakonikolas2021efficient,
  title={Efficient methods for structured nonconvex-nonconcave min-max optimization},
  author={Diakonikolas, Jelena and Daskalakis, Constantinos and Jordan, Michael I},
  booktitle={Proceedings of The 24th International Conference on Artificial Intelligence and Statistics},
  pages={2746--2754},
  year={2021},
  organization={PMLR}
}

@inproceedings{yoon2021accelerated,
  title={Accelerated algorithms for smooth convex-concave minimax problems with {$O(1/k^2)$} rate on squared gradient norm},
  author={Yoon, TaeHo and Ryu, Ernest K},
  booktitle={Proceedings of the 38th International Conference on Machine Learning},
  pages={12098--12109},
  year={2021},
  organization={PMLR}
}

@inproceedings{lee2021fast,
  title={Fast extra gradient methods for smooth structured nonconvex-nonconcave minimax problems},
  author={Lee, Sucheol and Kim, Donghwan},
  booktitle={Proceedings of the 35th International Conference on Neural Information Processing Systems},
  volume={34},
  pages={22588--22600},
  year={2021}
}

@article{xiao2011non,
  title={Non-smooth equations based method for $\ell_1$-norm problems with applications to compressed sensing},
  author={Xiao, Yunhai and Wang, Qiuyu and Hu, Qingjie},
  journal={Nonlinear Anal. Theor.},
  volume={74},
  number={11},
  pages={3570--3577},
  year={2011}
}

@article{popov1980modification,
  title={A modification of the Arrow-Hurwicz method for search of saddle points},
  author={Popov, Leonid Denisovich},
  journal={Math. Notes},
  volume={28},
  number={5},
  pages={845--848},
  year={1980}
}

@article{alon1999broad,
  title={Broad patterns of gene expression revealed by clustering analysis of tumor and normal colon tissues probed by oligonucleotide arrays},
  author={Alon, Uri and Barkai, Naama and Notterman, Daniel A and Gish, Kurt and Ybarra, Suzanne and Mack, Daniel and Levine, Arnold J},
  journal={Proc. Natl. Acad. Sci. USA},
  volume={96},
  number={12},
  pages={6745--6750},
  year={1999}
}

@article{golub1999molecular,
  author = {T. R. Golub  and D. K. Slonim  and P. Tamayo  and C. Huard  and M. Gaasenbeek  and J. P. Mesirov  and H. Coller  and M. L. Loh  and J. R. Downing  and M. A. Caligiuri  and C. D. Bloomfield  and E. S. Lander },
  title = {Molecular classification of cancer: class discovery and class prediction by gene expression monitoring},
  journal={Science},
  volume={286},
  number={5439},
  pages={531--537},
  year={1999}
}

@article{chang2011libsvm,
  title={LIBSVM: A library for support vector machines},
  author={Chang, Chih-Chung and Lin, Chih-Jen},
  journal={ACM Trans. Intell. Syst. Technol.},
  volume={2},
  number={3},
  pages={1--27},
  year={2011}
}

@article{kim2008interior,
  title={An interior-point method for large-scale $\ell_1$-regularized least squares},
  author={Kim, Seung-Jean and Koh, Kwangmoo and Lustig, Michael and Boyd, Stephen and Gorinevsky, Dimitry},
  journal={IEEE J. Sel. Top. Signal Process.},
  volume={1},
  number={4},
  pages={606--617},
  year={2008}
}

@article{pang1986inexact,
  title={Inexact Newton methods for the nonlinear complementarity problem},
  author={Pang, Jong-Shi},
  journal={Math. Program.},
  volume={36},
  number={1},
  pages={54--71},
  year={1986}
}

@book{bauschke2011convex,
  title={Convex Analysis and Monotone Operator Theory in Hilbert Spaces},
  author={Bauschke, Heinz H and Combettes, Patrick L},
  year={2011},
  publisher={Springer}
}

@article{tibshirani1996regression,
  title={Regression shrinkage and selection via the lasso},
  author={Tibshirani, Robert},
  journal={J. R. Stat. Soc. Ser. B},
  volume={58},
  number={1},
  pages={267--288},
  year={1996}
}

@article{wright2015coordinate,
  title={Coordinate descent algorithms},
  author={Wright, Stephen J},
  journal={Math. Program.},
  volume={151},
  number={1},
  pages={3--34},
  year={2015}
}

@article{bandeira2016sharp,
  title={Sharp nonasymptotic bounds on the norm of random matrices with independent entries},
  author={Bandeira, Afonso S. and van Handel, Ramon},
  journal={Ann. Probab.},
  volume={44},
  number={4},
  pages={2479--2506},
  year={2016}
}
\end{document}